%% file: SL.tex
\documentclass[reqno,12pt]{amsart}
\usepackage{amssymb,amsmath}
\usepackage{amsthm}
\usepackage{enumerate}
\usepackage[colorlinks]{hyperref}
\usepackage{booktabs}
\newtheorem{theorem}{Theorem}
\newtheorem{proposition}[theorem]{Proposition}
\newtheorem{lemma}[theorem]{Lemma}
\newtheorem{definition}[theorem]{Definition}

\newtheorem{remark}[theorem]{Remark}
\theoremstyle{definition}
\newtheorem{example}[theorem]{Example}

\newcommand{\R}{{\mathbb R}}
\newcommand{\Z}{{\mathbb Z}}
\newcommand{\N}{{\mathbb N}}
\newcommand{\Q}{{\mathbb Q}}

\newcommand{\CB}{{\mathcal B}}

\newcommand{\CH}{{\mathcal H}}

\newcommand{\CL}{{\mathcal L}}
\newcommand{\CM}{{\mathcal M}}

\newcommand{\CV}{{\mathcal V}}
\newcommand{\la}{{\langle}}
\newcommand{\ra}{{\rangle}}
\newcommand{\codim}{\operatorname{codim}}

\def\binomial(#1,#2){\binom{#1}{#2}}
\def\mult(#1,#2){\binom{#1}{#2}}   
\renewcommand{\a}{{\mathfrak{a}}}
\renewcommand{\b}{{\mathfrak{b}}}
\renewcommand{\c}{{\mathfrak{c}}}

\newcommand{\e}{{\mathrm{e}}}

\renewcommand{\k}{{\mathfrak{k}}}
\renewcommand{\t}{{\mathfrak{t}}}
\newcommand{\p}{{\mathfrak{p}}}
\newcommand{\q}{{\mathfrak{q}}}
\newcommand{\lattice}{\Lambda}
\newcommand{\co}{{k_0}}

\usepackage{ifthen}
\newcommand{\DeclareBracket}[3]{
  \newcommand{#1}[2][]{%
  \ifthenelse%
  {\equal{##1}{}}%
  {\left#2##2\right#3}%
  {\csname ##1l\endcsname#2##2\csname ##1r\endcsname#3}}}    
\DeclareBracket\charfun[]
\usepackage{pgf,ifpdf,graphics,xcolor}

\newenvironment{inputlist}
  {\begin{enumerate}[\quad\rm({I}$_\bgroup 1\egroup$)]}
  {\end{enumerate}}
\newenvironment{outputlist}
  {\begin{enumerate}[\quad\rm({O}$_\bgroup 1\egroup$)]}
  {\end{enumerate}}

\let\saveexample\example
\let\saveendexample\endexample
\def\example{\bgroup\small\saveexample}
\def\endexample{\saveendexample\egroup}

\def\vectwo(#1,#2){\left(%
    \!\begin{smallmatrix}#1\\#2\end{smallmatrix}\!\right)}
    
\title[]{Intermediate Sums on Polyhedra:
  Computation and Real Ehrhart Theory
}
\author{V. Baldoni}
\address{Velleda Baldoni: Dipartimento di Matematica, Universit\`a degli studi di  Roma ``Tor Vergata'',
Via della ricerca scientifica 1, I-00133, Italy}
\email{baldoni@mat.uniroma2.it}
\author{N. Berline}
\address{Nicole Berline: Centre de Math\'ematiques Laurent Schwartz, \'Ecole Polytechnique, 91128 Palaiseau Cedex, France}
\email{nicole.berline@math.polytechnique.fr}

\author{M. K\"oppe}
\address{Matthias~K\"oppe:  Department of
  Mathematics, University of California,
  Davis, One Shields Avenue, Davis, CA, 95616, USA}
\email{mkoeppe@math.ucdavis.edu}
\author{M. Vergne}
\address{Mich\`ele Vergne: Universit\'e Paris 7 Diderot, Institut Math{\'e}matique de
Jussieu, 16 rue Clisson 75013 Paris} \email{vergne@math.jussieu.fr}

\usepackage{rcs}
\RCS $Revision: 259 $
\date{November 27, 2010 (Revision \RCSRevision)}

\begin{document}
\maketitle

\begin{abstract}
  We study intermediate sums, interpolating between integrals and discrete
  sums, which were introduced by A.~Barvinok [\emph{Computing the {E}hrhart quasi-polynomial of a rational
    simplex}, Math.~Comp.~\textbf{75} (2006), 1449--1466]. 
  For a given semi-rational polytope~$\p$ and a rational subspace~$L$,
  we integrate a given polynomial function~$h$ over all lattice slices of the
  polytope ~$\p$ parallel to the subspace~$L$ and sum up the integrals.  We
  first develop an algorithmic theory of parametric intermediate generating functions.
  Then we study the Ehrhart theory of these intermediate sums, that is, the
  dependence of the result as a function of a dilation of the polytope.  We
  provide an algorithm to compute the resulting Ehrhart quasi-polynomials in
  the form of explicit step polynomials.  These formulas are naturally valid
  for real (not just integer) dilations and thus provide a direct approach to
  real Ehrhart theory.
\end{abstract}

{\small
 \tableofcontents}

\section{Introduction}

\bigskip

 Let $\p$ be a rational polytope in $V=\R^d$ and
$h(x)$ a polynomial function on $V$. A classical problem is to
compute the sum of values of $h(x) $ over the set of integral points
of $\p$,
$$
S(\p,h)=\sum_ {x\in \p\cap \Z^d}h(x).
$$
The  sum $S(\p,h)$ has generalizations, the
\emph{intermediate} sums $S^L(\p,h)$, where $L\subseteq V$ is a
rational vector subspace, introduced by
Barvinok~\cite{newbar}.  They interpolate between the discrete
sum $S(\p,h)$ and the integral $\int_\p h(x)\,\mathrm dx$. For a polytope
$\p\subset V$ and a polynomial $h(x)$
$$
S^L(\p,h)= \sum_{x} \int_{\p\cap (x+L)} h(y)\,\mathrm dy,
$$
where the summation index $x$ runs over the projected lattice in
$V/L$. In other words, the polytope $\p$ is sliced along affine
subspaces parallel to $L$ through lattice points and the integrals
of $h$ over the slices are added up. For $L=V$, there is only one
term and $S^V(\p,h)$ is just the integral of $h(y)$ over $\p$,
while, for $L=\{0\}$, we recover the discrete sum $S(\p,h)$. As in
the discrete case, a powerful method is to consider the
\textit{intermediate generating function}
\begin{equation}
S^L(\p)(\xi)= \sum_{x} \int_{\p\cap (x+L)} \e^{\la
\xi,y\ra}\,\mathrm dy,
\end{equation}
for $\xi\in V^*$.
 If $\p$ is a polyhedron, not necessarily compact, the generating
function still makes sense as a meromorphic function. By writing  a
polyhedron as the sum of  its cones at vertices (Brion's theorem),
we need only study the case where $\p$ is an affine  cone.

It is then natural to turn to the Ehrhart theory of the intermediate
sums~$S^L$, that is, the study of the intermediate sum~$S^L(t\p,h)$ of a
dilated polytope~$t\p$ as a function of the dilation parameter~$t$.
It turns out that, just like in the classical case, the Ehrhart
function $t\mapsto S^L(t\p,h)$ is a \emph{quasi-polynomial}, that is, 
a function of the form
\begin{equation}\label{eq:ehrhart-intro}
S^L( t\p,h)=\sum_{m=0}^{d+M}E_m^L(\p,h;t)\,t^m,
\end{equation}
where the coefficients $E^L_m(\p,h;t)$ depend only on $t \bmod q$, where $q$
is an integer  such that $q\p$ has lattice vertices.

\medbreak
The main results of this article are:
\begin{enumerate}[1.]\item 
  a polynomial time algorithm for the computation of the
  intermediate generating function of a simplicial affine cone, when the
  slicing space has fixed codimension, Theorem \ref{th:short_formula},
\item a polynomial time algorithm for the computation of the
  weighted intermediate sum $S^L(\p,h)$ of a simple polytope $\p$ (given by
  its vertices), and the
  corresponding Ehrhart quasipolynomial $t\mapsto S^L(t\p,h)$, when the
  slicing space has fixed codimension and the weight depends only on a fixed
  number of variables, or has fixed degree, Theorem~\ref{th:Ehrhart}.
\end{enumerate}
The main new feature of this article is the use of a decomposition
of any simplicial cone into simplicial cones which have a face
parallel to the given space  $L$, the decomposition being given by a
closed formula, Theorem \ref{brion-vergne-decomp}. This
decomposition is borrowed from \cite{Brion1997residue} where it
plays a key role in the study of partition functions.  We prove that
this decomposition is done by a polynomial algorithm, if the
codimension of $L$ is fixed.

Thus, the computation of the intermediate generating function of a
simplicial affine cone is reduced to the case where $L$ is parallel
to a face of the cone, and thus the generating function factors into a
discrete sum and an integral.  This case has been already studied in
\cite{BBDKV-2010}. The output of the algorithm is a ``short formula''
where both~$\xi$ and the vertex~$s$ of the cone appear as  symbolic
variables (see Examples \ref{ex:unimodular_cone},
\ref{ex:standard_cone} and~\ref{example_simplex3dim}). Once a short
formula for the generating function is available, we follow the
method of \cite{BBDKV-2010} for the computation of the weighted
Ehrhart quasi-polynomial of a simple polytope.  The method uses the fact that 
the intermediate generating function decomposes into meromorphic terms of
homogeneous $\xi$-degree. The output is again a
short formula where the dilating parameter~$t$ appears as a symbolic
variable inside \textit{step polynomials}, see
Example~\ref{example:Ehrhart_divided_square}.

The results remains valid for non-negative real values of $t$.  Thus our
method provides a direct proof of a real Ehrhart theorem
(Theorem~\ref{thm:ehrhart-nonalgo}). This extension of 
Ehrhart theory to real dilation factors has been studied by E.~Linke
\cite{linke:rational-ehrhart} for the classical case.  It is even more natural
for the intermediate valuations, as one of the cases is $L=V$. Then
$S^L(t\p,h)$ is the integral of $h$ over the $t\p$, the result being clearly a
polynomial formula valid for any non-negative real value of $t$.  \smallbreak

The present manuscript together with the article~\cite{BBDKV-2010} also provide
the foundation for  
a future work, in which the \textit{weighted Barvinok patched sums}
$\sum_L\lambda(L)S^L(\p,h)$ for a simple polytope will be studied. 

\section{Notations and basic facts}

\subsection{Rational and semi-rational polyhedra}
We consider  a \emph{rational vector space} $V$ of dimension $d$, that is
to say a finite dimensional real vector space with a lattice denoted
by  $\lattice$.   We will need to consider subspaces and quotient
spaces of $V$, this is why we cannot just let $V=\R ^d$ and
$\lattice = \Z^d$. An element $v\in V$ is called \emph{rational} if
$nv\in\lattice$ for some non-zero  integer $n$. The set of rational
points in $V$ is denoted by $V_\Q$. A subspace $L$ of $V$ is called
\emph{rational} if $L\cap \lattice $ is a lattice in $L$. If $L$ is a
rational subspace, the image of $\lattice$ in $V/L$ is a lattice in
$V/L$, so that $V/L$ is a rational vector space. The image of
$\lattice$ in $V/L$ is called the \emph{projected lattice}. It is denoted
by $\lattice_{V/L}$.  A rational space $V$, with lattice $\lattice$,
has a canonical Lebesgue measure $dx=\mathrm dm_\Lambda(x)$, for which
$V/\lattice$ has measure $1$.

A \emph{convex  polyhedron} $\p$ in $V$ (we will simply say
 \emph{polyhedron}) is, by definition, the intersection of a finite number of
closed half spaces bounded by   affine hyperplanes. If the
hyperplanes are rational, we say that the polyhedron is \emph{rational}. If
the hyperplanes have rational directions, we say that the polyhedron
is \textit{semi-rational}.
 For instance, if $\p\subset V $ is a
rational polyhedron, $t$~is a real number and $s$~is any point in~$V$, then
the dilated polyhedron $t\p$ and the translated polyhedron 
$s+\p$ are semi-rational. Unless stated otherwise, a polyhedron will
be assumed to be semi-rational.

In this article, a cone is a polyhedral cone (with vertex $0$) and
an affine cone is a translated set $s+\c$ of a cone $\c$. A  cone
$\c$ is called \emph{simplicial} if it is  generated by independent
elements of $V $. A simplicial cone~$\c$ is called \emph{unimodular} if it
is generated by independent integral vectors $v_1,\dots, v_k$ such
that $\{v_1,\dots, v_k\}$ can be completed to an integral basis of
$V$. An affine cone $\a$ is called \emph{simplicial} (resp.\ \emph{simplicial
unimodular}) if the associated cone is.

A \emph{polytope} $\p$ is a  compact polyhedron.
 The set of vertices of $\p$ is denoted by $\CV(\p)$.
 For each vertex $s$, the cone of feasible directions at $s$  is denoted by
$\c_s$.

We denote by $\CL(V)$ the space  of functions on $V$ which is
generated by indicator functions $\charfun{\q}$ of polyhedra $\q$
which contain lines.

\subsection{Valuations and generating functions}
\begin{definition}
Let $M$ be a vector space. An $M$-valued \emph{valuation} $\Phi$ is a  map
from the set of polyhedra $ \p\subset V$ to the vector space $M$
such that whenever the indicator functions $\charfun{\p_i}$ of a
family of polyhedra $\p_i$ satisfy a linear relation $\sum_i r_i
\charfun{\p_i}=0$, then the elements $\Phi(\p_i)$ satisfy the same
relation
$$
\sum_i r_i \Phi(\p_i)=0.
$$
\end{definition}

\begin{definition}
We denote by $\CH(V^*)$ the ring of holomorphic functions defined
around $0\in V^*$. We denote by $\CM(V^*)$ the ring of meromorphic
functions defined around $0\in V^*$  and by $\CM_{\ell}(V^*)\subset
\CM(V^*)$ the subring consisting of  meromorphic functions
$\phi(\xi)$ which can be written as a quotient of a holomorphic
function by  a product of linear forms.
\end{definition}

The intermediate generating functions of polyhedra which are studied
here are important examples of functions in $\CM_{\ell}(V^*)$.

\begin{proposition}\label{valuationSL}
Let $L\subseteq V$ be a rational subspace. There exists a unique
valuation  $S^L$ which to every semi-rational polyhedron $\p\subset
V$ associates a meromorphic function $S^L(\p)\in \CM(V^*)$ so that
the following properties hold:

\begin{enumerate}[\rm(a)]
\item 
 If $\p$ contains a line, then $S^L(\p)=0$.

\item
  \begin{equation}\label{SL}
    S^L(\p)(\xi)= \sum_{x\in \lattice_{V/L}} \int_{\p\cap (x+L)} \e^{\la
      \xi,y\ra} \,\mathrm dm_L(y),
  \end{equation}
  for every $\xi\in V^*$ such that the above sum converges.

\item For every point $s\in \lattice + L$, we have
$$
S^L(s+\p)(\xi) = \e^{\la \xi,s\ra}S^L(\p)(\xi).
$$
\end{enumerate}
\end{proposition}
When necessary  we will indicate the lattice and use the notation
$S^L(\p,\lattice)$.

\begin{remark} The fact that
$S^L(\p) $ is actually  an element of $\CM_{\ell}(V^*)$ will be a
consequence of  the explicit computations of the next section.
\end{remark}

For $L=\{0\}$, we recover the valuation $S$ given by
$$
S(\p)(\xi)= \sum_{x\in \p\cap \lattice}\e^{\la \xi,s\ra},
$$
provided this sum is convergent. For $L=V$,  $S^V(\p)(\xi)$ is the
integral
$$
I(\p)(\xi)=\int_\p \e^{\la \xi,x\ra}\,\mathrm dx,
$$
 if $\p$ is full dimensional, and $S^V(\p)=0$
otherwise.
 If $\p$ is
compact, the meromorphic function $S^L(\p)(\xi)$ is actually regular
at $\xi=0$, and its value for $\xi=0$ is the valuation
$E_{L^\perp}(\p)$ considered by Barvinok \cite{newbar}. The proof is
entirely analogous to the cases  $L= \{0\}$ and $L=V$, see for
instance Theorem 3.1 in \cite{barvinok:99} and we omit it.
\begin{remark}[Rational vs.\ semi-rational polyhedra]
  Although \cite{barvinok:99} and other classical texts deal only with
  rational polyhedra, their proofs extend easily to semi-rational ones. For
  instance, by summing geometric series, one obtains immediately the
  generating function of a unimodular cone with a real vertex.
\end{remark}
Here
are some simple examples. For $t\in\R$ we denote by $\{ t\}$  the
fractional part of $t$, i.e.,  $0\leq \{ t\}<1$ and $t-\{ t\}\in\Z$, and by
$\lceil t\rceil$ the ceiling of~$t$, i.e., $\lceil t\rceil=t + \{-t\}$.
\begin{example} \label{ex:unimodular_cone}For any $t \in \R$, we have
$$
S(t+\R_{\geq 0})(\xi)=\frac{\e^{\lceil t\rceil \xi}}{1-\e^\xi} = \e^{t
\xi} \frac{\e^{\{ - t\} \xi}}{1-\e^\xi}.
$$
More generally, let $\c \subset \R^d$ be a unimodular cone with
primitive edge generators $v_1,\dots,v_d$. Let $s$ be any point in
$\R^d$, with $s=\sum_{i=1}^d s_i v_i$. Then
$$
S(s+\c)(\xi)=\prod_{i=1}^d \frac{\e^{\lceil s_i\rceil
\langle\xi,v_i\rangle}} {1-\e^{\langle\xi,v_i\rangle}} = \e^{\langle
s,\xi \rangle} \prod_{i=1}^d \frac{\e^{\{- s_i \}
\langle\xi,v_i\rangle}} {1-\e^{\langle\xi,v_i\rangle}}.
$$
\end{example}
\begin{example}\label{ex:standard_cone}
Let $\c\subseteq\R^2$ be the first quadrant  and $L=\R e_1$. We write
$s=(a,b)\in \R^2$. We have
\begin{multline*}
S^L(s+ \c)(\xi_1,\xi_2)=\sum_{n=\lceil b\rceil}^\infty \int_a
^{+\infty} \e^{n\xi_2+x\xi_1}\,\mathrm dx =\biggl(\sum_{n=\lceil b\rceil}^\infty
\e^{n\xi_2}\biggr)\cdot\frac{-\e^{a\xi_1}}{\xi_1}\\
=\frac{\e^{\lceil b\rceil \xi_2}}{1-\e^{\xi_2}}\cdot\frac{-\e^{a
\xi_1}}{\xi_1}= \e^{\langle s,\xi\rangle}\frac{\e^{\{- b\}
\xi_2}}{1-\e^{\xi_2}}\cdot\frac{-1}{\xi_1}.
\end{multline*}
\end{example}

\begin{example}
Let $\p\subset \R^2$ be the triangle  with vertices
$(0,0)$, $(1,0)$, $(0,1)$. Let $L=\R (1,0)$. Straightforward calculations
give  $S^L (\p)=\frac{1-\e^{\xi_1}}{\xi_1}$, $
I(\p)(\xi)=\frac{1}{\xi_1-\xi_2}\bigl(\frac{1-\e^{\xi_1}}{\xi_1}-\frac{1-\e^{\xi_2}}{\xi_2}\bigr)
$, and   $S(\p)(\xi)=1+\e^{\xi_1}+\e^{\xi_2}$. These three functions
are indeed analytic.
\end{example}
A consequence of the valuation property is the following fundamental
theorem. It follows from the Brion--Lawrence--Varchenko decomposition of a
polytope into the supporting cones at its vertices \cite{Brion88}.
\begin{theorem}[Brion's theorem] \label{brion}Let $\p$ be a polytope with set of vertices $\CV(\p)$. For each
vertex $s$, let $\c_s$ be the cone of feasible directions at $s$.
Let $L\subseteq V$ be a rational subspace. Then
$$
S^L(\p)=\sum_{s\in \CV(\p)}S^L(s+\c_s).
$$
\end{theorem}

\section{Intermediate generating function  for a simplicial  affine
cone}\label{section:case-simplicial-cone}

\subsection{Where the slicing space is parallel to a face}
\label{subsection:slicing-parallel}

Let $\c\subset V $ be a simplicial cone with integral generators
$(v_j, j=1,\ldots,d)$, and let $s\in V$.  Let us recall the
expression for   the intermediate generating function
$S^L(s+\c)(\xi)$, when the slicing subspace $L$ is generated by some
of the $v_j$'s.

For $I\subseteq \{1,\ldots, d\}$, we denote by $L_I$ the linear span
of the vectors $(v_i,i\in I)$. We denote the complement of $I$ in
$\{1,\ldots, d\}$ by $I^c$. We have $V=L_I\oplus L_{I^c}$. For $x\in
V$ we denote the components
 by
$$
x=x_I + x_{I^c}.
$$
Thus we identify the quotient $V/L_{I}$ with $L_{I^c}$ and   we
denote the projected lattice by $\lattice_{I^c}\subset L_{I^c}$.
Note that $L_{I^c} \cap \lattice \subseteq \lattice_{I^c}$, but the
inclusion is strict in general.
\begin{example}
Let $v_1=(1,0)$, $v_2=(1,2)$, $I=\{2\}$. The projected
lattice~$\Lambda_{I^c}$ on $L_{I^c}=\R v_1$ is $\Z \frac{v_1}{2}$.  See
Figure~\ref{fig:projected-lattice}. 
\begin{figure}[t]
  \centering
  \ifpdf
  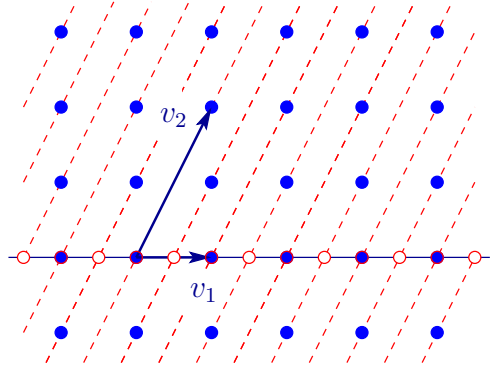
   \else
   \input{projected-lattice.pstex_t}
   \fi
  \caption{The projected lattice~$\Lambda_{\{1\}}$}
  \label{fig:projected-lattice}
\end{figure}
\end{example}
 We denote by   $\c_{I}$ the cone  generated by the vectors  $v_j$, for $j\in
I$ and by $\b_I$  the parallelepiped
 $\b_I=\sum_{i\in I}[0,1\mathclose[\, v_i$.
   Similarly
 we denote by $\c_{I^c}$ the cone generated by  the vectors  $v_j$, for $j\in
I^c$ and $\b_{I^c}=\sum_{i\in I^c}[0,1\mathclose[\, v_i$. The
projection of the cone $\c$ on $V/L_I=L_{I^c}$ identifies with
$\c_{I^c}$. Note that a generator $v_i$, $i\in I^c$, may be
non-primitive for the projected lattice $\lattice_{I^c}$, even if it
is primitive for $\lattice$, as we see in the previous example. We
write $s=s_I+s_{I^c}$.

First let us recall from~\cite{BBDKV-2010} how the intermediate generating
function~$S^{L_I}(s+\c,\Lambda)$ decomposes as a product.  This already
appeared in Example \ref{ex:standard_cone}.
\begin{proposition}  \label{petite-somme}
\cite{BBDKV-2010}. The intermediate sum for the full cone $s+\c$
breaks up into the product
\begin{equation}\label{petite-somme-eqn}
 S^{L_I}(s+\c,\Lambda)(\xi) =
 S(s_{I^c}+\c_{I^c},\lattice_{I^c})(\xi)\, I(s_I+\c_I,L_I\cap
 \lattice)(\xi).
\end{equation}
\end{proposition}
We remark that in the above equation, the
function $S(s_{I^c}+\c_{I^c},\lattice_{I^c})(\xi)$ is a meromorphic
function on the space $(L_{I^c})^*$. The integral
$I(s_I+\c_I,L_I\cap \lattice)(\xi)$ is a meromorphic function on the
space $(L_I)^*$. We consider both as functions on $V^*$ through the
decomposition $V=L_I\oplus L_{I^c}$.

\subsection{Cone decomposition with respect to a linear subspace}

Let $L$ be a linear subspace of $V$. Let $\c$ be a full dimensional
simplicial cone in $V$ with generators $w_1,\ldots,w_d$. The main
theorem of this section, Theorem \ref{brion-vergne-decomp}, gives a
signed decomposition of $\c$ into a family of full dimensional
simplicial cones $\c_\sigma$, where each $\c_\sigma$ has a face
parallel to~$L$. More precisely, this decomposition is
 a relation between indicator functions of cones, modulo the
space $\CL(V)$  of functions on $V$ which is generated by indicator
functions $\charfun{\q}$ of polyhedra $\q$ which contain lines. This
decomposition  is borrowed from \cite{Brion1997residue}.  The proof
relies on the following simple lemma which has its own interest. The
extreme cases, $\dim L=1$ or $\codim L=1$,  of Theorem
\ref{brion-vergne-decomp} follow easily from this lemma
(Proposition \ref{dim1-codim1}). The general case is  proven by
induction on either $\dim L$ or $\codim L$.

Given vectors $u_1,\ldots, u_p$ $\in\R^d$, we denote the cone
$\sum_{i=1}^p \R_{\geq 0} u_i$ by $\c(u_1,\ldots,u_p)$ and its
relative interior $\sum_{i=1}^p \R_{>0} u_i$ by
$\tilde{\c}(u_1,\ldots,u_p)$.

\begin{lemma}\label{openconesdec}
Let $u_1,\ldots, u_{d}$ be a basis of $V$ and let
$u_{d+1}=-\sum_{k=1}^{d}u_k$. For $i=1,\ldots, d+1$, define
semi-open cones
$$
\k_i=\c(u_1,\ldots,u_{i-1})+\tilde{\c}(u_{i+1},\ldots,u_{d+1}).
$$
Then $V$ is the disjoint union of the $\k_i$ for $i=1,\ldots, d+1$.
\end{lemma}
\begin{proof}
Let $x=t_1 u_1+\ldots+t_d u_d$. If $t_k\geq 0$ for all $k$, then
$x\in \k_{d+1}$. Otherwise, let $i$ be the largest
 index with $ 1\leq i\leq d$, for which the minimum $\min_{1\leq k \leq d} t_k$ is reached.
 Writing $u_i=-\sum_{1\leq k\leq d+1, k\neq i}u_k$, we obtain
 $$
 x=\sum_{1\leq k\leq d, k\neq i}(t_k-t_i)u_k -t_i u_{d+1}.
 $$
As $t_k-t_i\geq 0$ if $1\leq j \leq d $ and $t_j-t_i>0$ if $i<j\leq
 d$, and $t_i<0$, we see that $x\in \k_i$.

 Let us show that  $\k_i\cap \k_j$ is empty  if $i<j$. Let $x\in \k_i$,
  $x=\sum_{1\leq k\leq d+1, k\neq i}t_k  u_k$. Writing $u_j=-\sum_{1\leq k\leq d+1, k\neq
 j}u_k$, we obtain
 $$
 x=\sum_{1\leq k\leq d, \, k\notin \{i,j\}}(t_k-t_j)u_k -t_j u_{i}.
 $$
As $t_j>0$, it follows that $x\notin \k_i$.
\end{proof}

From the lemma, we  deduce other interesting cone decompositions.
The first one is the well-known stellar decomposition.
\begin{proposition}[\textbf{Stellar decomposition}]
Let $w_1,\ldots, w_d$ be a basis of $V$. Let $r\leq s$ and
$$
v = w_1+\cdots +w_r-(w_{s+1}+\cdots + w_d).
$$
   Let $\c$ be the
cone generated by $w_i$'s for $1\leq i\leq d$ and let $\k$ be the
cone generated by $\c$ and the vector $v$.

\noindent\textup{(a)} Define the semi-open cones $\k_i$ by
\begin{eqnarray*}
&& \k_i = \tilde{\c}(w_1,\ldots,w_{i-1})+ \c(w_{i+1},\ldots,w_d,v) \mbox{    for } 1\leq i \leq r, \mbox{  if } r\geq 1,\\
&& \k_i= \c(w_1,\ldots , w_{i-1})+\tilde{\c}( w_{i+1}, \ldots ,w_d,v
)
 \mbox{    for } s+1\leq i \leq d, \mbox{  if } s < d.
\end{eqnarray*}
Then:

\textup{(a.1)} If $r\geq 1$,  $\k$ is
the disjoint union of the $\k_i$ for $1\leq i\leq r$.

\textup{(a.2)}  $\k\setminus \c$ is the disjoint union of the
 $\k_i$ for $s+1\leq i \leq d $.  (If $s=d$, then $\k\setminus \c=\emptyset$.)
\medskip

\noindent\textup{(b)}  For  $ i\in\{1,\ldots,r\}\cup \{s+1,\ldots,d\}$,  let
$$\c_i=\c(w_1,\ldots,w_{i-1},
w_{i+1},\ldots,w_d,v).
$$
Then
$$
\charfun{\c}- \sum_{i=1}^r \charfun{\c_i}+ \sum_{i=s+1}^d
\charfun{\c_i} \equiv 0
$$ modulo indicator functions of lower dimensional cones.
\end{proposition}
\begin{proof}
Let us prove (a.2) first. By quotienting out by the subspace
generated by $w_1,\ldots,w_s $, we can assume that $s =0$, thus
$v=-(w_1+\ldots+w_d)$. In this case, we just apply  Lemma
\ref{openconesdec} with $u_k=w_k$ and $u_{d+1}=v$. As the cone
$\k_{d+1}$ of Lemma \ref{openconesdec} is just $\c$, we obtain (a.2).

Let us now assume that $l\geq 1$. We write
$$
w_1= v+ w_d + \cdots +w_{s+1}-(w_r+ \cdots + w_2).
$$
We apply (a.2) with $\bar{w}_1=v$, $\bar{w}_2=w_d$ etc., and
$\bar{v}=w_1$. The corresponding semi-open cones are $
\bar{\k}_1=\c(v,w_d,\ldots,w_{r+1})+\tilde{\c}(w_{r-1},\ldots,w_2,w_1)=\k_r
, \dots, \bar{\k}_r=\c(v,w_d,\ldots,w_{2})=\k_1 $, thus we obtain
(a.1).

For  $ i\in\{1,\ldots,r\}\cup \{s+1,\ldots,d\}$, the (closed) cone
$\c_i$ differs from the semi-open cone $\k_i$ by the union of  some
faces of lower dimension, therefore (b) follows from (a) and
inclusion-exclusion relations.
\end{proof}

Next, we obtain decompositions of any simplicial cone into full-dimensional
cones which have a given edge, or a given facet. Recall 
that $\CL(V)$ is the linear span of indicator functions of cones
which contain straight lines.
\begin{proposition}\label{dim1-codim1}
 Let $w_1,\ldots, w_d$ be a basis of $V$ and let $\c$ be the
cone generated by $w_i$'s for $1\leq i\leq d$.  Let $0\leq r\leq
s\leq d+1$.

\noindent \textup{a)} Let

\begin{equation}\label{extraedge}
v = w_1+\cdots +w_r-(w_{s+1}+\cdots + w_d).
\end{equation}

Then we have the following relation between indicator functions
of cones.
\begin{multline}\label{eq:cone-decomposition-dim1}
\charfun{\c} \equiv \sum_{i=1}^{r}(-1)^{i-1}
\charfun{\c(-w_1,\ldots,-w_{i-1},w_{i+1},\ldots,w_d,v)}\\
+
\sum_{i=s+1}^{d}(-1)^{d-i}\charfun{\c(w_1,\ldots,w_{i-1},-w_{i+1},\ldots,-w_d,-v)}
\;\;\;\mbox{     mod } \CL(V).
\end{multline}

\noindent \textup{b)} Let $L\subset V$ be a hyperplane. Assume  that $w_i\in
L$ if and only if  $r+1\leq i\leq s$, that the $w_i$'s lie on one side of $L$
for $1\leq i\leq r$   and on the other side  for $s+1\leq i\leq d$.
For $ i\in\{1,\ldots,r\}\cup \{s+1,\ldots,d\}$, denote by $\rho_i$
the projection $V\to L$ parallel to $w_i$. Then we have the
following relation between indicator functions of cones.
\begin{equation}\label{eq:cone-decomposition-codim1}
\charfun{\c} \equiv \sum_{i=1}^{r}\charfun{\R_+ w_i+ \rho_i(\c)}   -
\sum_{i=s+1}^{d}\charfun{\R_+ (-w_i)+ \rho_i(\c)}\;\;\; \mbox{ mod }
\CL(V).
\end{equation}
\end{proposition}
\begin{example} \label{ex:Brion-Vergne-decomp}$\c=\c( w_1, w_2)$, $v=w_1+w_2$.
Then
$$
\charfun{\c}=-\charfun{\c(-w_2,v)}+\charfun{\c(w_1,v)}+\charfun{\c(w_2,-w_2,w_1)}.
$$
Indeed $\charfun{\c(w_2,-w_2,w_1)}-\charfun{\c}$ is the indicator
function of the quadrant $\c(w_1, -w_2)$ minus that of the  line $\R
w_1$. This is also the case for
$\charfun{\c(-w_2,v)}+\charfun{\c(w_1,v)}$.
See also Figure~\ref{fig:brion-vergne-2d}.
\end{example}

\begin{example}\label{example_conedecomp}
We consider the $3$-dimensional cone
$$
\c=\c((-1,0,0),(-1,2,0),(-1,0,3)).
$$
\noindent i) First let $L$  be the subspace with basis $ (0,1,1)$.
Formula \eqref{eq:cone-decomposition-dim1} gives
\begin{multline*}
 \charfun{\c}\equiv
\charfun{\c((-1, 0, 0), (-1, 2, 0), (0, 1, 1))}-\charfun{\c((-1, 0,
0), (1, 0, -3), (0, 1, 1))} \\
+ \charfun{\c((-1, 2, 0), (-1, 0, 3), (0, -1, -1))}.
\end{multline*}
\noindent ii) Now, let $L$ be the subspace with basis $(1,0,0),
(0,1,1)$. Formula \eqref{eq:cone-decomposition-codim1} gives
$$
\charfun{\c}\equiv \charfun{\c((-1, 2, 0), (-1, 0, 0), (-5, 6,
6))}-\charfun{\c ((1, 0, -3), (-1, 0, 0), (-5, 6, 6))}.
$$
 This example will be continued in Example \ref{example_simplex3dim}.
\end{example}
\begin{proof}[Proof of Proposition \ref{dim1-codim1}]
Let us prove a). By quotienting out by the subspace generated by
$w_{r+1},\ldots,w_s$, we can assume that $r=s$. We write
\eqref{extraedge} as
$$
w_r + \cdots +w_1 -v -w_d -\cdots  -w_{r+1} =0.
$$
By applying  Lemma \ref{openconesdec}
 with $u_1=-w_{r},\ldots ,u_{r+1}=-v,\ldots,  u_{d+1}=w_{r+1}$,
we see that the whole space $V$ is the disjoint union of the
semi-open cones
\begin{eqnarray*}
\k_1&=&\tilde{\c}(w_{r-1},\ldots,-v,-w_d ,\ldots  ,-w_{r+1}), \\
\ldots&& \\
\k_r&=& \c(w_r,\ldots,w_2)+\tilde{\c}(-v,-w_d ,\ldots ,-w_{r+1}),\\
\k_{r+1}&=&\c(w_r,\ldots,w_1)+\tilde{\c}(-w_d ,\ldots  ,-w_{r+1}),\\
\k_{r+2}&=&\c(w_r,\ldots,w_1,-v)+\tilde{\c}(-w_{d-1} ,\ldots  ,-w_{r+1}),\\
\ldots&& \\
\k_{d+1}&=&\c(w_{r},\ldots,w_1,-v,-w_d ,\ldots  ,-w_{r+2}).
\end{eqnarray*}
For any basis $u_1,\ldots, u_d$, we have
\begin{multline*}
  \charfun{(\c(u_1,\ldots, u_k)+\tilde{\c}(u_{k+1},\ldots, u_d))}
  \\\equiv
  (-1)^{d-k}\charfun{\c(u_1,\ldots, u_k,-u_{k+1},\ldots, -u_d)} \;\;\; \mbox{
    mod } \CL(V).
\end{multline*}
Therefore we obtain
\begin{eqnarray*}
\charfun{\k_1}&\equiv &(-1)^d \charfun{\c(-w_{r-1},\ldots,-w_1,v,w_d
,\ldots
,w_{r+1})}\\
\ldots&& \\
\charfun{\k_r}&\equiv &(-1)^{d-r} \charfun{\c(w_r,\ldots,w_2,v,w_d ,\ldots  ,w_{r+1})}\\
\charfun{\k_{r+1}}  &\equiv &
(-1)^{d-r-1}\charfun{\c(w_r,\ldots,w_1,w_d
,\ldots,w_{r+1})}\\
\ldots&& \\
\charfun{\k_{d+1}}&= &\charfun{\c(w_{r},\ldots,w_1,-v,-w_d ,\ldots
,-w_{r+2})}.
\end{eqnarray*}
Moreover these indicator functions add up to $\charfun{V}\equiv 0
\mbox{ mod } \CL(V)$. Thus we have proven a).

Let us show that b) is just the dual of a stellar decomposition. Let
$\alpha\in V^*$ be such that $\ker \alpha=L$, and such that
$\alpha(w_i)>0$  for $1\leq i\leq r$   and $\alpha(w_i)< 0$ for
$s+1\leq i\leq d$.  Denote the dual cone by $\c^\circ\subset V^*$,
$$
\c^\circ=\{\, \xi\in V^* : \langle\xi,x\rangle\geq 0  \mbox{ for all  } x\in
\c\,\}.
$$
Let $(\gamma_1,\ldots \gamma_d)$ be the dual basis of $(v_1,\ldots,
v_d)$. Thus $\c^\circ=\c(\gamma_1,\ldots,\gamma_d)$. From the stellar
decomposition of the  cone $\c^\circ$ with respect to the extra edge
$\alpha$, we see that the expression
$$
\charfun{\c^\circ}- \sum_{ i\in\{1,\ldots,r\}\cup
\{s+1,\ldots,d\}}\charfun{\c(\gamma_1,\ldots,
\hat{\gamma_i},\ldots,\gamma_d,\alpha)}
$$
is   a linear combination of indicator functions of lower
dimensional cones. If a linear identity holds for indicator
functions of cones, the same identity holds for the indicator
functions of their duals, (see for instance \cite{barvinok:99},
Corollary 2.8). Moreover, the dual of a lower dimensional cone
contains a line, therefore by taking indicator functions of dual
cones we obtain an equality mod $\CL(V)$.  (This  ``duality trick''
goes back to \cite{Brion88}). There remains to compute the dual cone
$(\c(\gamma_1,\ldots, \hat{\gamma_i},\ldots,\gamma_d,\alpha))^\circ$. We
see easily that if $\langle\alpha,w_i\rangle>0$,  then
$$(\c(\gamma_1,\ldots, \hat{\gamma_i},\ldots,\gamma_d,\alpha))^\circ=
\R_+ w_i+\rho_i(\c),
$$
 while if $\langle\alpha,w_i\rangle< 0$,
then
$$(\c(\gamma_1,\ldots,
\hat{\gamma_i},\ldots,\gamma_d,\alpha))^\circ= \R_+ (-w_i)+\rho_i(\c).
$$
Thus we have proven b).
\end{proof}
\medskip

As we will see now, Proposition \ref{dim1-codim1} consists of
particular cases of the next result, Brion--Vergne decomposition. Let
$L$ be a linear subspace of $V$. Let $\c$ be a simplicial cone in
$V$ with generators $w_1,\ldots,w_d$. We consider the subsets
$\sigma\subset \{1,\ldots, d\}$ such that the vectors $w_j$, for
$j\in\sigma$, form a supplementary basis to $L$ in $V$. In other
words  we have a direct sum decomposition
$$
V=\bigoplus_{j\in\sigma} \R w_j \oplus L.
$$
The corresponding projection  on $L$ is denoted by $\rho_\sigma\colon
V\to L $. The set of these $\sigma$ is denoted by $\CB(\c,L)$.

\begin{remark}\label{CB-polynomial-in-fixed-codim}
  At most $\binomial(d, \co) =
  O(d^{\co})$ such sets~$\sigma$ exist.  Thus, if $\codim L$ is a fixed
  constant~$\co$, this is a polynomial quantity, and the sets~$\sigma\in \CB(\c,L)$
  can be enumerated by a
  straightforward algorithm in polynomial time.  (In practice, reverse search
  \cite{avis-fukuda-1996:reverse-search} would be used for this enumeration.)
\end{remark}

A vector $a$ in $V/L$ is said to be \emph{generic} with
respect to the pair $(\c,L)$ if $a$ does
not belong to any hyperplane generated by vectors among the projections of
$w_1,\ldots,w_d$. In other words, for any $\sigma\in\CB(\c,L)$, we
have
$$
a= \sum_{j\in\sigma} a_{\sigma,j} (w _j \bmod L),
$$
where $a_{\sigma,j}\neq 0$ for every $j\in\sigma$.

\begin{figure}[t]
  \noindent\hspace*{-1em}
 \begin{minipage}[c]{.35\linewidth}%
   \scalebox{.75}{\ifpdf
  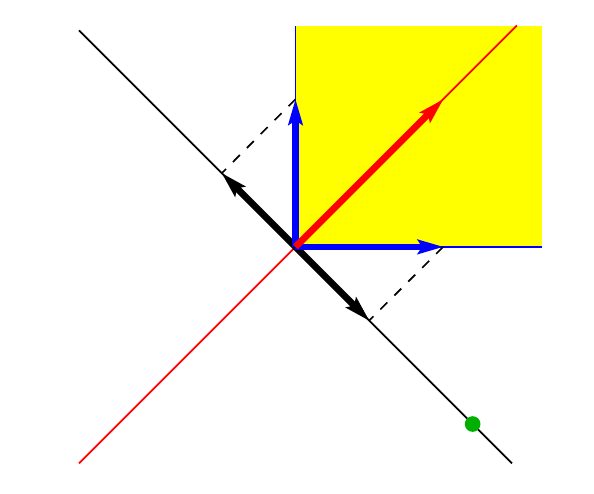
   \else
   \input{brionvergne-2d.pstex_t}
   \fi}%
 \end{minipage}%
 {\Huge$\equiv$}\hspace*{-1em}%
 \begin{minipage}[c]{.31\linewidth}%
   \scalebox{.75}{%
     \ifpdf
  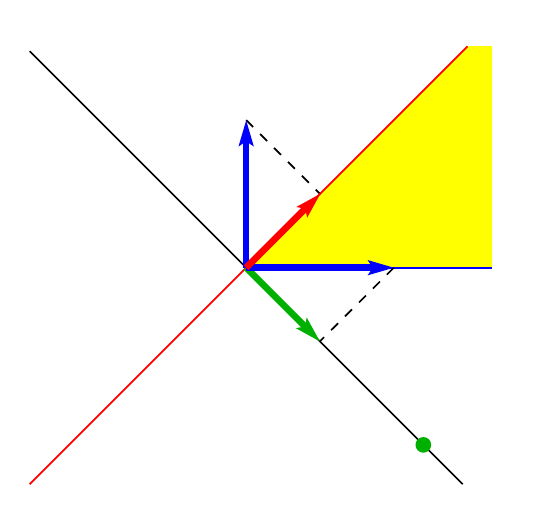
   \else
   \input{brionvergne-2d-cone1.pstex_t}
   \fi}%
 \end{minipage}%
 {\Huge$-$}%
 \begin{minipage}[c]{.35\linewidth}%
   \scalebox{.75}{%
     \ifpdf
  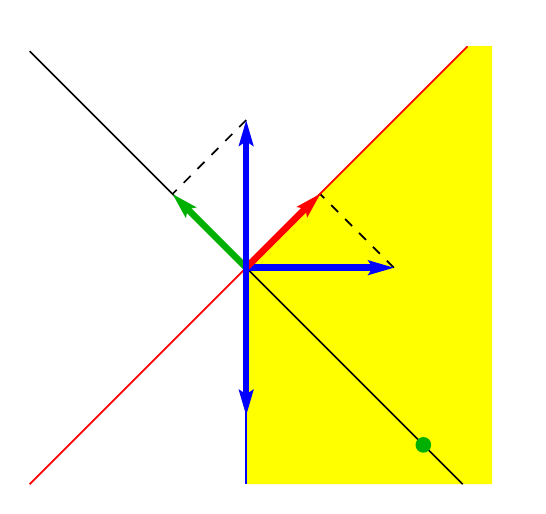
   \else
   \input{brionvergne-2d-cone2.pstex_t}
   \fi%
   }
 \end{minipage}%
 \caption{Brion--Vergne decomposition of a cone~$\c$ into cones with a face
   parallel to the subspace~$L$, modulo cones with lines.  The vectors in the
   quotient $V/L$ determine the signs~$\epsilon_{\sigma,j}$.}
 \label{fig:brion-vergne-2d}
\end{figure}

\begin{theorem}[Brion--Vergne decomposition, \cite{Brion1997residue}, Theorem 1.2] \label{brion-vergne-decomp}
Let $L$ be a linear subspace of  $V$. Let $\c$ be a full dimensional
simplicial cone in $V$ with generators $w_1,\ldots,w_d$. Fix a
vector $a\in V/L $ that is generic with respect to the pair $(\c,L)$
and belongs to the projection of~$\c$ on~$V/L$. For
$\sigma\in\CB(\c,L)$, let
$$
a= \sum_{j\in\sigma} a_{\sigma,j}(w_j \bmod L).
$$
Let $ \epsilon_{\sigma,j}$ be the sign of $ a_{\sigma,j}$ and
$\epsilon(\sigma)=\prod _{j\in\sigma }\epsilon_{\sigma,j}$. Denote
by $\c_\sigma\subset V$ the cone with edge generators
$$
\epsilon_{\sigma,j}w_j \mbox{ for } j\in\sigma, \mbox{ and }
\rho_\sigma(w_k) \mbox{ for } k\notin\sigma.
$$
Then we have the following relation between indicator functions
of cones.
\begin{equation}\label{eq:cone-decomposition}
\charfun{\c}\equiv \sum_{\sigma\in\CB(\c,L)}\epsilon(\sigma)
\charfun{\c_\sigma} \mbox{ mod } \CL(V).
\end{equation}
\end{theorem}
The theorem is illustrated in Figure~\ref{fig:brion-vergne-2d}.

\begin{remark}
As explained in \cite{Brion1997residue}, the case where $\dim L=1$
reduces to  case a) of Proposition \ref{dim1-codim1}. Let us show
that the case where $\codim L= 1 $ reduces to case b). A vector $a$
in $V/L$ is generic with respect to $(\c, L)$ if and only if
$a\neq0$. Up to renumbering, we can assume that $w_1, \ldots, w_r$
lie on the same side of $L$ as $a$, while $w_{s+1}, \ldots, w_d$ lie
on the other side. The set $\CB(\c,L)$ consists of the singletons
$\sigma_i=\{i\}$ such that $w_i\notin L$, ie  $i\in
\{1,\ldots,r\}\cup \{s+1,\ldots,d\}$. The sign
$\epsilon_{\sigma_i,i}$ is $+1$ if $i\in \{1,\ldots,r\}$ and $-1$ if
$i\in \{s+1,\ldots,d\}$, therefore Equation
\eqref{eq:cone-decomposition} becomes  precisely  Equation
\eqref{eq:cone-decomposition-codim1}.
\end{remark}

\begin{theorem}[Efficient construction of a Brion--Vergne decomposition]
  Fix a number~$k_0$.  There exists a polynomial-time algorithm for the
  following problem.  Given a rational simplicial cone~$\c$ and a rational
  subspace $L$ of~$V$ of codimension at most~$k_0$, compute a
  decomposition~\eqref{eq:cone-decomposition}.
\end{theorem}

\begin{proof}
  By Remark~\ref{CB-polynomial-in-fixed-codim}, we can enumerate the bases
  $\sigma\in \CB(\c,L)$ in polynomial time.  

  A rational generic vector~$a$ can be constructed in polynomial time as follows;
  this technique was already used in \cite{bar}.
  Fix one of the bases $\bar\sigma\in\CB(\c,L)$ as a basis of $V/L$.  
  Then the projections of~$w_1,\dots,w_d$ have rational coordinate
  vectors~$\bar w_1,\dots,\bar w_d \in\Q^{\co}$ with 
  respect to this basis. 

  Consider the \emph{moment curve} $\bar a(t) = (1, t, t^2,
  \dots, t^{\co-1})\in\R_+^{\co}$ for $t\geq0$.  
  For any $\sigma\in\CB(\c,L)$ and $j\in\sigma$, we have by Cramer's rule that
  the coefficient $a_{\sigma,j}(t)$ for $a=a(t)$ is nonzero 
  if and only if $f_{\sigma,j}(t) = \det(a(t); \bar w_i :
  i\in\sigma, i\neq j) \neq 0$. The functions $f_{\sigma,j}(t)$ are polynomial
  functions in~$t$ of degree at most~$\co(\co-1)$, which are not identically
  zero.  Hence each has at most $\co(\co-1)$ zeros.  
  Thus, by computing the coefficients $a_{\sigma,j}(t)$ for $M = \co(\co-1)
  |\CB(\c,L)| + 1$ different values for~$t$, we will find one~$t=\bar t$
  that is not a zero of any~$f_{\sigma,j}$.  Clearly this search can be implemented
  in polynomial time.  

  Then the vector $a = \sum_{i\in\sigma} \bar a_i(t) (w_i\bmod L) \in V/L$ is generic with respect
  to the pair~$(\c, L)$, and since $\bar a_i(t)\geq0$, it also lies in the
  projection of~$\c$ on~$V/L$.  (We could also use explicit deterministic constructions of
  generic vectors of polynomial encoding size, without search,
  similar to
  \cite{koeppe:irrational-barvinok,koeppe-verdoolaege:parametric}.)
\end{proof}

\begin{remark}[Relation to variants of Barvinok's decomposition]
  Barvinok's method to efficiently compute the discrete generating function
  $S(s+\c)(\xi)$ of a cone in fixed dimension decomposes a simplicial cone
  into cones with smaller indices (and ultimately into unimodular cones).
  The original method by Barvinok \cite{bar} used a primal decomposition and
  inclusion-exclusion to take care of lower-dimensional cones.  
  The dual of the stellar decomposition is used in the variant of Barvinok's
  method that was presented in~\cite{barvinok:99} and later implemented in
  \emph{LattE}~\cite{latte-1.2}; this avoids having to take care of lower dimensional cones.  
  Instead of dualizing, one can directly uses Formula
  \eqref{eq:cone-decomposition-dim1}, the case where $\dim L=1$.  We use this
  primal variant of Barvinok's decomposition in our implementation.   Other
  primal variants were studied in \cite{koeppe:irrational-barvinok} and
  \cite{koeppe-verdoolaege:parametric} and implemented in \emph{LattE macchiato}.
\end{remark}

\subsection{Complexity of the intermediate generating function}

We will use the following notations.
\begin{definition}
  Let
  \begin{equation}\label{bernoulli-genfun}
    T(t,x)=\e^{tx  }\frac{x}{1-\e^x} = -\sum_{n=0}^\infty B_n(t) \frac{x^n}{n!},
  \end{equation}
  where $B_n(t)$ are the Bernoulli polynomials.
\end{definition}
\begin{definition}\label{fractional_part} For $a$ and $q$ a positive integer, we denote
\begin{displaymath}
  \lfloor n\rfloor_q := q\bigl\lfloor \tfrac1q n\bigr\rfloor \in q\Z,\quad
  \{n\}_q := (n \bmod q) \in [0,q),
\end{displaymath}
which give the unique decomposition
\begin{displaymath}
  n = \lfloor n\rfloor_q + \{n\}_q.
\end{displaymath}
Thus $\lfloor a\rfloor_1=\lfloor a\rfloor $ and $\{a\}_1=\{a\}$ will
be the ordinary ``floor'' and ``fractional part'' of $a$.
\end{definition}

\begin{theorem}[Short formula
for~$S^L(s+\c)(\xi)$]\label{th:short_formula}

Fix a non-negative  integer $\co$. There exists a polynomial time
algorithm for the following problem. Given the following {input}:
\begin{inputlist}
\item a number $d$ in unary encoding,

\item a simplicial cone $\c = \c(v_1,\dots,v_d) \subset \R^d$,
  represented by the primitive vectors $v_1,\dots,v_d\in\Z^d$ in binary
   encoding,

 \item  a rational subspace $L\subseteq \R^d$ of codimension $\co $, represented by $d-\co$
  linearly independent vectors $b_1,\ldots , b_{d-\co}\in \Z^d$ in binary
  encoding,
\end{inputlist}
compute the following {output} in binary encoding:
\begin{outputlist}
\item a finite set~$\Gamma$,
\item for every $\gamma$ in $\Gamma$, integers~$\alpha^{(\gamma)}$ and
lattice  vectors~$w^{(\gamma)}_i\in \Z^d$ for $i=1,\dots,d$,
\end{outputlist}
which have the following properties.
\begin{enumerate}[\rm (1)]
\item
For every $\gamma$, the family $(w^{(\gamma)}_i)$, for
$i=1,\dots,d$, is a basis of $\R^d$ and if we denote the dual basis
by $(\eta_i^{(\gamma)})$, for $i=1,\dots,d$, then
$\eta_i^{(\gamma)}\in \Z^d$ for $i=1,\dots ,k_0$.
\item
For every $s\in \R^d$,  we have the following equality of
meromorphic functions of~$\xi$:
\begin{multline}\label{SL_shortformula}
  S^{L}(s+\c,\Z^d)(\xi)
  \\
  = \e^{\langle \xi,s\rangle} \sum_{\gamma\in \Gamma} \alpha^{(\gamma)}
    \prod_{i=1}^{k_0} T\bigl(\bigl\{ - \bigl \langle \eta_i^{(\gamma)},
    s\bigr\rangle\bigr\}, \bigl\langle \xi,w^{(\gamma)}_i\bigr\rangle \bigr) \cdot\frac{1} {\prod_{i=1}^d\bigl\langle
      \xi,w^{(\gamma)}_i\bigr\rangle}.
\end{multline}
\end{enumerate}
\end{theorem}
\begin{remark}
Consider the term corresponding to $\gamma\in\Gamma$ in
\eqref{SL_shortformula}.  Define
$$
s^{(\gamma)}=\sum_{i=1}^{k_0} \bigl\{-\bigl\langle \eta
_i^{(\gamma)},s\bigr\rangle\bigr\}\, w_i^{(\gamma)}.
$$
The point $s+s^{(\gamma)}$ is in $\bigoplus_{i=1}^{k_0} \Z
w_i^{(\gamma)} \oplus \bigoplus_{i=k_0+1}^d  \R w_i^{(\gamma)}$, and
formula~\eqref{SL_shortformula} reads also
\begin{multline}\label{SL_shortformula-sans-Todd}
  S^{L}(s+\c,\Z^d)(\xi)
=\\
\sum_{\gamma\in \Gamma} \alpha^{(\gamma)} \e^{\langle
\xi,s+s^{(\gamma)}\rangle}
     \frac{1} {\prod_{i=1}^{k_0} \bigl(1-\e^{\langle \xi, w_i^{(\gamma)}\rangle}\bigr)}
    \frac{1}{\prod_{i=k_0+1}^d \langle \xi, w_i^{(\gamma)}\rangle
      \vphantom{\e^{\langle \xi, w_i^{(\gamma)}\rangle}} 
    }.
  \end{multline}
Formula \eqref{SL_shortformula-sans-Todd}  is easier to grasp, while
\eqref{SL_shortformula} is used in the program, where  we  avoid all
the singular hyperplanes  $\langle \xi, w_i^{(\gamma)}\rangle=0$,
for $\gamma \in \Gamma$ and  $i=1,\dots,d$, by deforming   the
linear form $\xi$ into  $\xi +\epsilon$.
\end{remark}
\begin{proof}

Using Theorem \ref{brion-vergne-decomp}, we construct a signed
decomposition of~$\c$ into cones $\c^\sigma$ which have a face
parallel to~$L$, so that for any $s\in \R^d$ we have
\begin{equation}\label{coeff-1}
  \charfun{s+\c} \equiv \sum_\sigma \epsilon^\sigma \charfun{s +
    \c^\sigma}
  \quad\text{(modulo cones containing lines)},
\end{equation}
where $\epsilon^\sigma\in\{\pm1\}$. Here $\sigma$ runs through the
set $\CB(\c,L)$ of Theorem~\ref{brion-vergne-decomp}, which can be
enumerated in polynomial time by
Remark~\ref{CB-polynomial-in-fixed-codim}. This makes the
construction of the signed decomposition a polynomial-time
algorithm. As the valuation $S^L$ vanishes on cones which contain
lines, we get
$$
S^L(s+\c)(\xi)= \sum_{\sigma\in\CB(\c,L)}
\epsilon^\sigma S^L(s+ \c^\sigma)(\xi).
$$
Each cone $\c^\sigma$ is given by its edge generators $w_i^\sigma$ ,
$i=1,\dots,d$, which are such that $(w_i^\sigma)$, for
$i=k_0+1,\dots,d)$, is a basis of $L$.  Then we apply Theorem 28
[Short formula for~$S^{L_I}(s+\c,\Z^d)(\xi)$] of \cite{BBDKV-2010}
with input  $\c^\sigma$ and $I=\{k_0+1,\dots,d\}$. See Examples
\ref{ex:unimodular_cone} and \ref{ex:standard_cone}.
\end{proof}

\begin{example} [Continuation of Example
\ref{example_conedecomp}] \label{example_simplex3dim}

Again $\c$ is the cone $$\c((0,-2,0),(1,-2,0),(0,-2,3)).$$ Here is the
output of our Maple program computing the generating functions
$S^L(s+\c)(\xi)$ (Formula \eqref{SL_shortformula-sans-Todd}). We set
$s=(a,b,c)$ and
 $\xi=(x,y,z)$.

\noindent i) Let $L$  be the subspace with basis $ (0,1,1)$. In this
case, the three cones which occur in the decomposition of Example
\ref{example_conedecomp} have non-unimodular projections in $V/L$,
so that the projections are  further decomposed into unimodular
cones. This is why  we have  six terms in the expression of
$S^L(s+\c)$:
\begin{multline*}
S^L(s+\c)(x,y,z)\\
= {\rm e}^{ax+by+cz} \biggl( \frac {{\rm e}^{ \left\{ 3
a-b+c
 \right\} ( \frac{2}{5}y-\frac{3}{5}z) + \left\{ a \right\}(-x+3\,z ) }}
{( 1-{\rm e}^{\frac{2}{5}y-\frac{3}{5}z}) ( 1-{{\rm e}^{-x+3\,z}}) ( y+z
)}\qquad\qquad\qquad\qquad\\
+\frac { {\rm e}^{ \left\{ a \right\} ( -x+2\,y ) + \left\{ 2\,a+b-
c \right\}( -\frac{2}{5}y+\frac{3}{5}z) }}{ ( 1- {\rm e}^{-x+2 \,y})
( 1-{{\rm e}^{-\frac{2}{5}y+\frac{3}{5}z}}) (y+z) }
\\
 - \frac {{\rm e}^{- \left\{ 3\,a-b+c \right\} z+ \left\{ -a \right\}( x-3 z )}}{( 1-
{{\rm e}^{-z}} ) ( 1-{{\rm e}^{x-3\,z}}) ( -y -z) }
+\frac {{\rm e}^{- \left\{ -b+c \right\} z+ \left\{ -a
 \right\} x}}{( 1-{\rm e}^{-z}) ( 1-{\rm e}^x) ( -y-z) } \\
+  \frac {{\rm e}^{ \left\{ -b+c \right\} y-
 \left\{ a \right\} x}}{( 1-{\rm e}^y )( 1-{\rm e}^{-x} )( -y-z) }+
\frac {{\rm e}^{ \left\{ a \right\}  ( -x+2 y ) - \left\{ 2\,a+b-c
 \right\} y}}{ (1-{\rm e}^{-x+2\,y}) ( 1-
{\rm e}^{-y}) (-y-z) }\biggr).
\end{multline*}
\noindent ii) Let $L$  be the   subspace with basis $(1,0,0)$, $(0,1,1)$. Then:
\begin{multline*}
S^L(s+\c)(x,y,z)\\
= {\rm e}^{ax+by+cz} \biggl(
-6\,{\frac {{\e^{ \left\{ -b+c \right\} \left( -\frac{1}{2}x+ y
\right) }}}{ \left( 1-{{\rm e}^{-\frac{1}{2}x+y}} \right) x \left(
-5\,x+6 \,y+6\,z \right) }} \\
+6 \frac {{{\rm e}^{\left\{ -b+c
 \right\}  \left( \frac{1}{3}x-z \right) }}}{ \left( 1-{{\rm e}^{\frac{1}{3}x-z}}
 \right) x \left( -5\,x+6\,y+6\,z \right) }\biggr).
\end{multline*}
\end{example}

\renewcommand\t{T}  
\begin{example}\label{example:square}
We consider the four triangles $\t_i$ of Figure \ref{divided_square}
which subdivide the square $\q=[0, 4]\times [0, 4]$. 
\begin{figure}[t]
\begin{center}
  \ifpdf
  \includegraphics[width=2in]{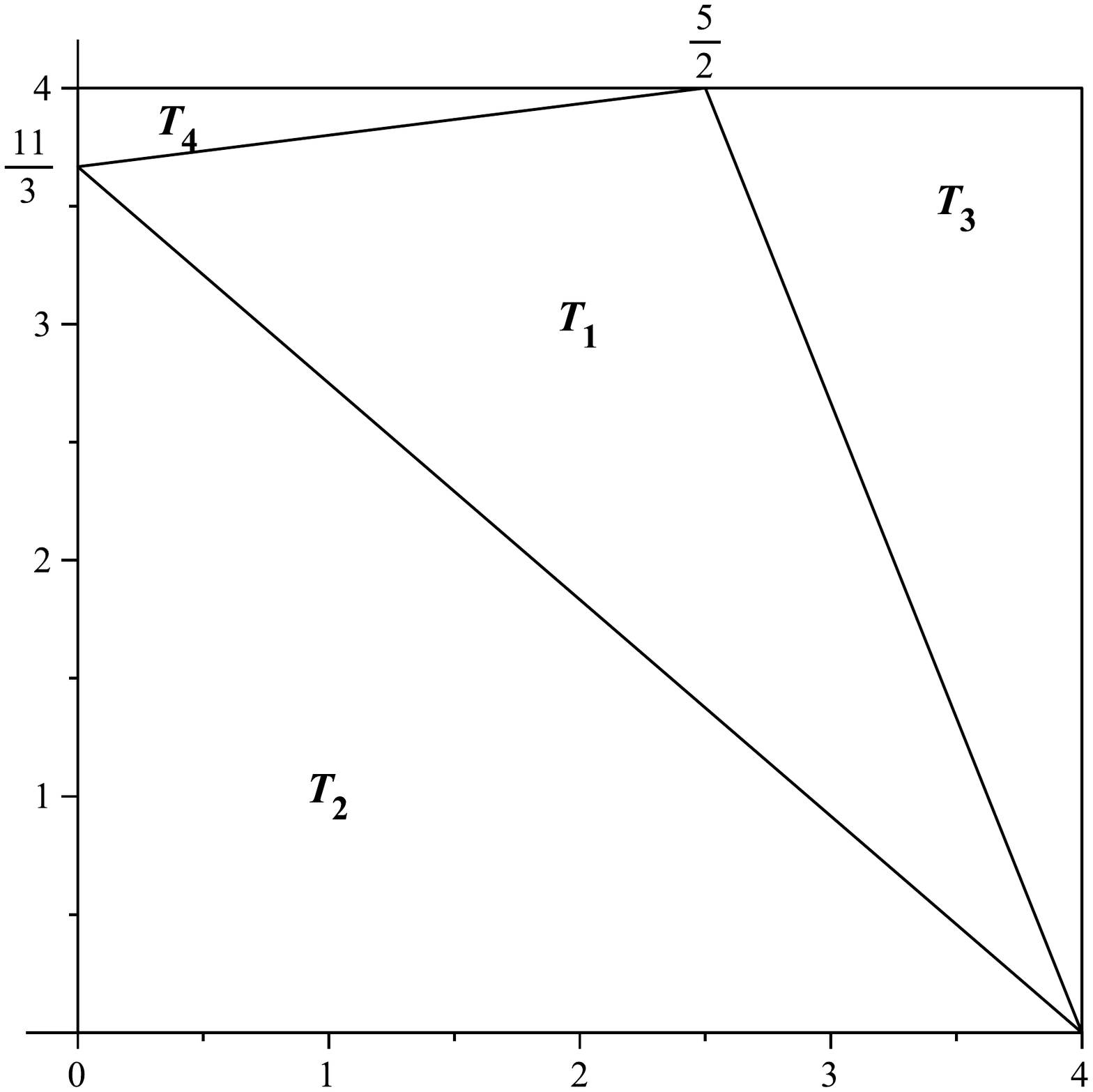}
  \else
  \includegraphics[width=2in]{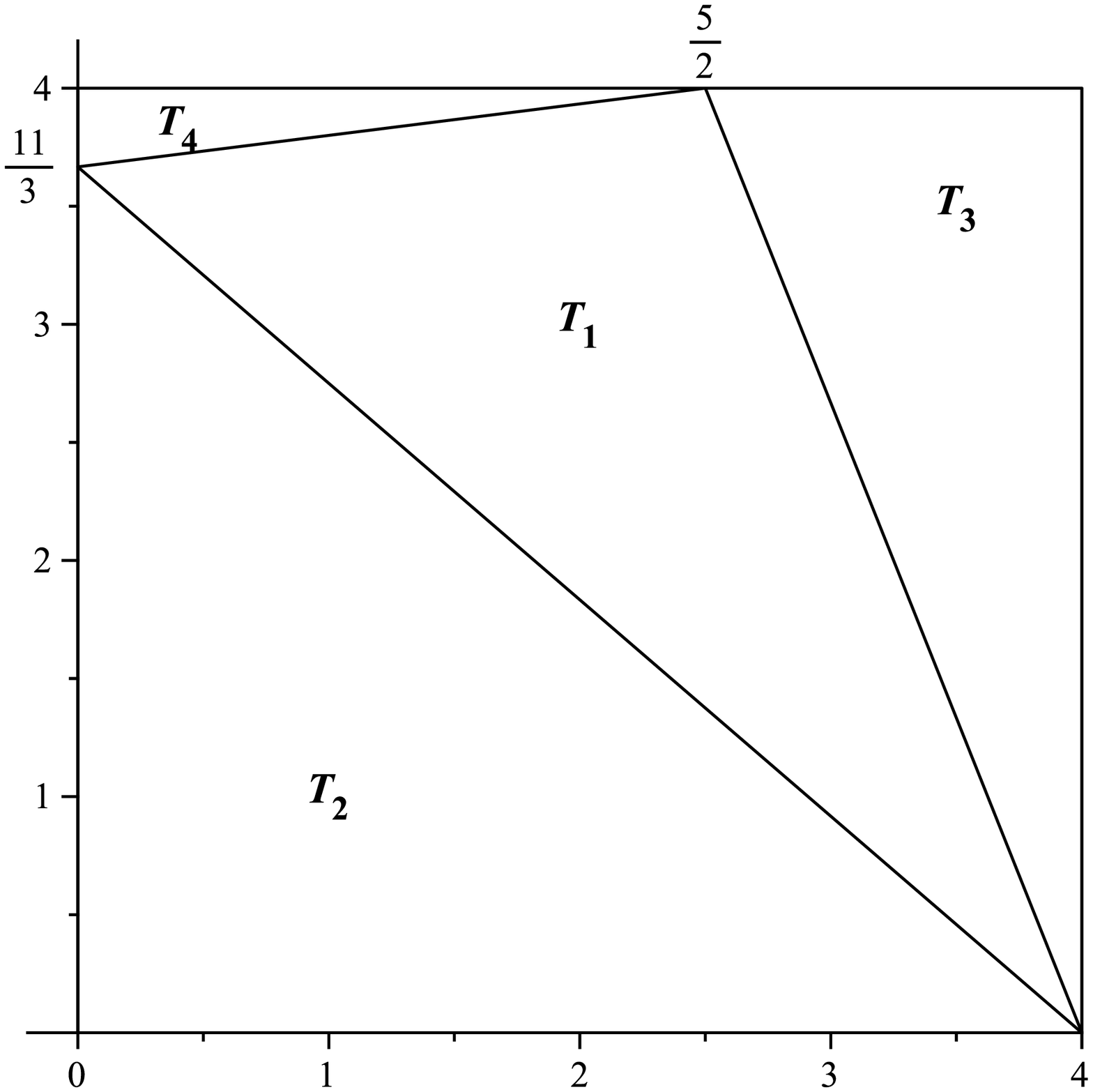}
  \fi
  \caption{The square~$\q$ subdivided into 4 triangles~$\t_i$.
    Lattice subspaces parallel to the subspace $L=\R(0,1)$ intersect the walls
    between the triangles at most in a point, and thus $S^L$ vanishes on the
    intersections of the triangles.}\label{divided_square}
\end{center}
\end{figure}%
\begin{table}[t]
  \caption{Intermediate valuations of the square and the 4 triangles}
  \label{tab:subdivided-square}
  \begin{center}
  \begin{tabular}{lcc}
    \toprule
    \multicolumn{1}{c}{$\p$} & Vertices~$\CV(\p)$ & \multicolumn{1}{c}{Intermediate valuation $S^L(\p)$} \\ 
    \midrule
    $\t_1$ & $\vectwo(4,0),\vectwo(\frac{5}{2},4),\vectwo(0,\frac{11}{3})$
    &
    \multicolumn{1}{l}{${\frac {{\e^{4\,x}}}{ \left( 1-{\e^{-x+\frac{8}{3}y}}
        \right) y}}
    - \frac{\e^{4\,x}} {\left( 1-{\e^{-x+{\frac {11}{12}}\,y}}
      \right) {y}}
    +{\frac {{\e^{3\,x+\frac{8}{3}y}}}{ \left( 1-{\e^{x-\frac{8}{3}\,y}}
        \right) y}}$}\\
    & &
    \multicolumn{1}{r}{${}-\frac{\e^{3\,x+{\frac {61}{15}}\,y}} {\left( 1-{
          \e^{x+\frac{2}{15}y}} \right) {y}}
    -\frac{\e^{\frac{11}{3}\,y}}{ \left( 1-{\e^{x-{\frac {11}{12}}\,y}}
      \right){y}}
    +{\frac {{ \e^{\frac{11}{3}\,y}}}{ \left( 1-{\e^{x+\frac{2}{15}y}}
        \right) y}} $ }\\
    $\t_2$ & $\vectwo(0,0),\vectwo(4,0),\vectwo(0,\frac{11}{3})$
    &
    $-{\frac {1}{ \left( 1-{\e^{x}} \right) y}}
    -{\frac {{\e^{4\,x}}}{ \left( 1-{\e^{-x}} \right) y}}
    +\frac{\e^{4\,x}} {\left( 1-{ \e^{-x+{\frac {11}{12}}\,y}} \right){y}}
    +\frac{\e^{\frac{11}{3} \,y}} {\left( 1-{\e^{x-{\frac {11}{12}}\,y}}
      \right){y}} $
    \\
    $\t_3$ & $\vectwo(\frac{5}{2},4),\vectwo(4,0),\vectwo(4,4)$
    &
    $-{\frac {{\e^{3\,x+\frac{8}{3}y}}}{ \left(
          1-{\e^{x-\frac{8}{3}y}}
        \right) y}}+{\frac {{\e^{3\,x+4\,y}}}{ \left( 1-{\e^{x}}
        \right) y}}-{\frac {{\e^{4\,x}}}{ \left( 1-{\e^{-x+\frac{8}{3}y}}
        \right) y}}+{\frac {{\e^{4\,x+4\,y}}}{ \left( 1-{\e^{-x}}
        \right) y}}
    $
    \\
    $\t_4$ & $\vectwo(\frac{5}{2},4),\vectwo(0,4),\vectwo(0,\frac{11}{3})$
    &
    ${\frac {{\e^{2\,x+4\,y}}}{ \left( 1-{\e^{-x}} \right) y}}
    -\frac{ \e^{2\,x+{\frac {59}{15}}\,y}} {\left(
        1-{\e^{-x-\frac{2}{15}y}}
      \right){y}}
    +{\frac {{\e^{4\,y}}}{ \left( 1-{\e^{x }} \right) y}}-{\frac
      {{\e^{\frac{11}{3}\,y}}}{ \left( 1-{\e^{x+2/15 \,y}} \right) y}} $\\
    \midrule
    $\q$ & $\vectwo(0,0),\vectwo(0,4), \vectwo(4,4),\vectwo(4,0)$& 
    \multicolumn{1}{c}{$ \tfrac{ \e^{4y}-1}{y}(1+\e^x+\e^{2x}+\e^{3x}+\e^{4x}) $}
    \\
    \bottomrule
  \end{tabular}
\end{center}
\end{table}%
Let $L=\R(0,1)$.
Let us apply the  inclusion-exclusion formula together with  the
valuation property of $S^L$. As $S^L$ vanishes on the edges and
vertices which occur in the intersections of  the triangles, it
follows that  the sum $\sum_{i=1}^4 S^L(\t_i)$ must be equal to
$S^L(\q)$. 

The example has been chosen so that the intermediate valuation~$S^L(\q)$ 
can also be computed directly by elementary means.  This allows us to verify
our methods on the example.  A direct computation gives
$$ S^L(\q)=\tfrac{ \e^{4y}-1}{y}(1+\e^x+\e^{2x}+\e^{3x}+\e^{4x}). $$
Table~\ref{tab:subdivided-square} shows the output of our Maple program that
implements the algorithm of the theorem for  $S^L(\t_i)$,
$i=1,\dots,4$.  The equality $\sum_{i=1}^4 S^L(\t_i)=S^L(\q)$ is readily checked on
these expressions.
\end{example}

\clearpage
\section{Ehrhart quasipolynomials for intermediate valuations}

Let us fix  a rational polytope $\p$ and a weight $h(x)$. It is well
known (``Ehrhart's Theorem'') that when we dilate $\p$ by an \emph{integer} $t\in \N$, the
function $t \mapsto S( t\p,h)$ is a quasi-polynomial function of~$t$ of degree
$d+M$ where $M$ is the degree of $h$, 
\begin{equation}\label{eq:ehrhart}
S( t\p,h)=\sum_{m=0}^{d+M}E_m(\p,h;t)\,t^m.
\end{equation}
The coefficients $E_m(\p,h;t)$ depend only on $\{t\}_q$  where $q$
is an integer  such that $q\p$ has lattice vertices.  When $\p$ is a lattice
polytope, then the coefficients $E_m(\p,h;t)$ do not depend on~$t$, and $t
\mapsto S( t\p,h)$ becomes a polynomial.

E.~Linke \cite{linke:rational-ehrhart} has proved that the expression
\eqref{eq:ehrhart} is still valid for \emph{real dilations} $t\in\R$.  This is the
generalized setting that we are going to use in this section.
Note that if we allow real dilations, then even for lattice polytopes~$\p$ 
we obtain a quasi-polynomial rather than a polynomial; the fractional part
function~$\{\cdot \}_q$ with $q=1$ appears in the expressions for the 
coefficients~$E_m(\p,h;t)$. 
For instance, the number of integers in the interval
$[0,t] = t [0,1]$ is $t+1 -\{t\}_1$.\smallbreak

We are going to extend Ehrhart's Theorem to the intermediate valuation
$S^L(\p)$. Indeed, one still has
\begin{equation}\label{eq:ehrhart-sl}
S^L( t\p,h)=\sum_{m=0}^{d+M}E_m^L(\p,h;t)\,t^m
\end{equation}
where the coefficients $E^L_m(\p,h;t)$ depend only on $\{t\}_q$.
Note that when  $L=V$, then $S^L(t\p,h)$ is polynomial, and coincides with the
integral of~$h$ on $t\p$.

 More precisely, we will show that the Ehrhart
coefficients  are step-polynomial functions of $t$, in the sense
of~\cite{verdoolaege-woods-2005}. In the next theorem, we prove these 
results and furthermore, we show that the coefficients
$E_m^L(\p,h;t)$ are given by short formulas, provided that in the
input $(L,\p,h)$, the subspace $L$ has fixed codimension, $\p$ is a
simple polytope given by its vertices,  and the weight  is a power
of a linear form $h(x)=\langle\ell,x\rangle^M$.

\begin{theorem}\label{th:Ehrhart}
For every fixed number ~$k_0 \in\N$, there exists a polynomial-time
algorithm
  for the following problem.

\noindent  Input:
\begin{inputlist}
\item a number~$ d\in\N$ in unary encoding, with $d\geq k$,
\item  a rational subspace $L\subseteq \R^d$ of codimension $\co $, represented by $d-\co$
  linearly independent vectors $b_1,\ldots , b_{d-\co}\in \Q^d$ in binary
  encoding,
\item a finite index set $\CV$,
\item a simple rational polytope~$\p$, given by the set~$\CV$ of its vertices, 
  rational vectors $ s_i \in \Q^d$ in binary encoding,
\item a rational vector $\ell \in \Q^d$ in binary encoding,
\item a number~$M \in \N$ in unary encoding.
\end{inputlist}
\noindent   Output:
\begin{outputlist}
\item  a finite  index set $\Gamma$,
\item  polynomials $f^{\gamma,m} \in \Q[r_1,\ldots, r_{k_0} ]$ and numbers $\zeta_i^{\gamma,m}\in \Z$,
$q_i^{\gamma,m} \in \N$ for $\gamma \in  \Gamma$, $i=1,\dots ,k_0$
and $m = 1,\dots , d+M$,
\end{outputlist}
such that for $t\in \R$ we have $S^L(
t\p,\langle\ell,x\rangle^M)=\sum_{m=0}^{d+M}E_m(t)\,t^m $, where
$E_m(t)$ is given by the step-polynomial
$$
E_m(t)=\sum_{\gamma\in \Gamma}f^{\gamma,m}
\bigl(\{\zeta_1^{\gamma,m}t\}_{q_1^{\gamma,m}},\dots,
\{\zeta_{k_0}^{\gamma,m}t\}_{q_{k_0}^{\gamma,m}}\bigr).
$$
\end{theorem}

\begin{remark}
Similar results hold for a more general weight $ h(x)$, as in
\cite{BBDKV08} for  integrals  over a simplex. One can assume that
the weight is given as a polynomial in a fixed number $D$ of linear
forms,
$h(x)=f(\langle\ell_1,x\rangle,\dots,\langle\ell_D,x\rangle)$, or
has a fixed degree $D$.
\end{remark}

\begin{proof}[Proof of Theorem~\ref{th:Ehrhart}]
As first observed in \cite{Baldoni-Berline-Vergne-2008}, the path
from generating functions to  Ehrhart quasi-polynomials  relies on
the following important property of functions $\phi(\xi)\in
\CM_{\ell}(V^*)$: such a function  has a unique expansion into
homogeneous rational functions
$$
\phi(\xi)= \sum_{m\geq m_0}\phi_{[m]}(\xi).
$$
If $P$ is a homogeneous polynomial on $V^*$ of degree $p$, and $D$ a
product of $r$ linear forms, then $\frac{P}{D}$ is an element in
$\CM_{\ell}(V^*)$  homogeneous of degree $m=p-r$. For instance,
\smash{$\frac{\xi_1}{\xi_2}$} is homogeneous of degree $0$. On this example,
we observe that a function in $\CM_\ell(V^*)$ which has no negative
degree terms need not be  analytic.

For a fixed $t\in \R$, the polynomial function  $\xi\mapsto S^L(
t\p,\frac{\langle\xi,x\rangle^M}{M!})$ is the term of degree $M$ of
of the holomorphic function $S^L( t\p)(\xi)$. Let $\c_s$ be the
supporting cone of $\p$ at the vertex $s$. By Brion's theorem
applied to the semi-rational polytope $t\p$, we write $S^L(
t\p)(\xi)$ as the sum of the intermediate generating functions of
the cones at the vertices $ts$, $s\in \CV$, of the dilated polytope
$t\p$. The crucial point is that the cone $\c_s$ does not change
when the polytope is dilated.
\begin{equation}\label{apply_Brion}
S^L( t\p)(\xi) =\sum_{s\in \CV}S^L( t s +\c_s )(\xi).
\end{equation}
Let $q_s$ be the smallest integer such that $q_s s$  is a lattice
point.  Then  $(t-\{t\}_{q_s})s$ is a lattice point, therefore we
have
\begin{eqnarray*}
S^L( ts +\c_s )(\xi)&=&\e^{\langle\xi,(t-\{t\}_{q_s})s\rangle}S^L(
\{t\}_{q_s}s +\c_s )(\xi)\\
&=& \e^{t\langle\xi,s\rangle} \e^{-\langle\xi,\{t\}_{q_s}s\rangle}S^L(
\{t\}_{q_s}s +\c_s )(\xi).
\end{eqnarray*}

Expanding $\e^{t\langle\xi,s\rangle}$ in powers of $t$ and looking
for the coefficient of $t^i$, we obtain
$$
S^L\bigl( t\p,\tfrac{\langle\xi,x\rangle^M}{M!}\bigr)=\sum_{i\geq 0} t^i
\sum_{s\in
\CV}\tfrac{\langle\xi,s\rangle^i}{i!}\bigl(\e^{-\langle\xi,\{t\}_{q_s}s\rangle}S^L(
\{t\}_{q_s}s +\c_s )(\xi)\bigr)_{[M-i]}.
$$
Thus we must analyze the term of $\xi$-degree $M-i$ of
$$
\e^{-\langle\xi,\{t\}_{q_s}s\rangle}S^L( \{t\}_{q_s}s +\c_s )(\xi).
$$
We use Formula \eqref{SL_shortformula} and consider the summands
indexed by $\gamma\in\Gamma$, one at a time. From then on, the proof
is entirely similar to the proof of Theorem 37 in \cite{BBDKV-2010},
and we omit it.
\end{proof}
\begin{example}[Example
\ref{example:square}, continued]\label{example:Ehrhart_divided_square}
Table~\ref{tab:subdivided-square-ehrhart} shows the output of our Maple program for the Ehrhart
quasi-polynomials of the four triangles of Example
\ref{example:square} with respect to the weight $h=1$. Note that the
dilating parameter~$t$ is real.
\begin{table}[t]
  \caption{Intermediate Ehrhart quasi-polynomials of the square and the 4 triangles}
  \label{tab:subdivided-square-ehrhart}
\begin{center}
\begin{tabular}{ll}
  \toprule
  \multicolumn{1}{c}{$\p$} 
  & \multicolumn{1}{l}{Ehrhart quasi-polynomial $S^L(t\p,1)$} \\
  \midrule
  \\[-1ex]
  $\t_1$ 
  &
  ${\frac
    {21}{4}}\,\mbox{\boldmath${t}^{2}$} -{\frac {7}{8}}\, \left( \{ 4\,t \}_{1} \right) ^{2} -{\frac {7}{10}}\,\{
  -5\,t \}_{2} +{\frac {7}{8}}\,\{ 4\,t \}_{1}
  +{\frac {7}{20}}\, \left( \{ -5\,t \}_{2}  \right) ^{2}$
  \\[2ex]
  $\t_2$ 
  &
  ${\frac {22}{3}}\,\mbox{\boldmath${t}^{2}$}
  +{\frac {11}{6}}\,\mbox{\boldmath${t}$}
  -{\frac {11}{24}}\, \left( \{ 4\,t \}_{1} \right) ^{ 2}+{\frac
    {11}{24}}\,\{ 4\,t \}_{1} 
  $
  \\[2ex]
  $\t_3$ 
  &
  $3\,\mbox{\boldmath${t}^{2}$}
  + \bigl(2 -4\,\{ 4\,t \}_{1} \bigr)\, \mbox{\boldmath${t}$}
  -\frac{1}{3} \left( \{ -5\,t \}_{2}
  \right) ^{2}+\frac{2}{3}\{ -5\,t \}_{2}
  +\frac{4}{3} \left( \{ 4\,t \}_{1}  \right) ^{2} -\frac{4}{3}\{ 4\,t \}_{1}
  $
  \\[2ex]
  $\t_4$ 
  &
  ${\frac {5}{12}}\,\mbox{\boldmath${t}^{2}$}
  +\frac{1}{6} \,\mbox{\boldmath${t}$}
  -{\frac {1}{60}}\, \left( \{ 5\,t \}_{2}
  \right) ^{2 }+\frac{1}{30}\,\{ 5\,t \}_{2} $\\[1ex]
  \midrule
  \\[-1.3ex]
  $\q$ 
  &
  $16\,\mbox{\boldmath${t}^{2}$}
  +\bigl(4-4\,\{ 4\,t \}_{1}\bigr) \,\mbox{\boldmath${t}$}$
   \\[1ex]
  \bottomrule
  \end{tabular}
\end{center}
\end{table}%
We obtain
\begin{align*}
\sum_{i=1}^4 S^L(t\t_i)&= 16\,{t}^{2}+4\,t-4\,t\{ 4\,t \}_{1}
-\tfrac{1}{30}\,\{ -5\,t \}_{2}\\[-1ex]
&\qquad + \tfrac{1}{30}\,\{ 5\,t \}_{2}
 +  {\tfrac {1}{60}}\, \left( \{ -5\,t \}_{2}
\right) ^{2} -{\tfrac {1}{60}}\, \left( \{ 5\,t \}_{2}  \right) ^{2}\\
&=16\,{t}^{2}+4\,t-4\,t\{ 4\,t \}_{1},
\end{align*}
using a relation between fractional parts,
$\{x\}^2-\{-x\}^2 =\{x\}-\{-x\}$,
for the simplification.
Indeed, a direct calculation gives the same answer.

\end{example}

Setting aside considerations of efficient computation, we can prove the
following theorem.

\begin{theorem}[Real Ehrhart Theorem]\label{thm:ehrhart-nonalgo}
  Let $\p\subseteq\R^d$ be a rational polytope, $h$ be a polynomial function
  of degree~$M$, and $L\subseteq\R^d$ be a rational subspace of codimension $\co $.  Then $t\mapsto
  S^L( t\p,h)$ is a quasi-polynomial function of~$t\in\R_{>0}$.  More precisely,
  \begin{equation}\label{eq:ehrhart-sl-again}
    S^L( t\p,h)=\sum_{m=0}^{d+M}E_m^L(\p,h;t)\,t^m,
  \end{equation}
  where $E_m^L(\p,h;t)$ is a step-polynomial of the form
  $$
  E_m(t)=\sum_{\gamma\in \Gamma}f^{\gamma,m}
  \bigl(\{\zeta_1^{\gamma,m}t\}_{q_1^{\gamma,m}},\dots,
  \{\zeta_{k_0}^{\gamma,m}t\}_{q_{k_0}^{\gamma,m}}\bigr)\quad\text{for $t\in\R_{>0}$,}
  $$
  where $\Gamma$ is a finite index set, with polynomials $f^{\gamma,m} \in
  \Q[r_1,\ldots, r_{k_0} ]$ and numbers 
  $\zeta_i^{\gamma,m}\in \Z$, $q_i^{\gamma,m} \in \N$ for $\gamma \in \Gamma$,
  $i=1,\dots ,k_0$ and $m = 1,\dots , d+M$.
\end{theorem}
\begin{proof}
  The polytope~$\p$ is no longer assumed to be simple, so we compute a
  triangulation of each of the vertex cones and use inclusion--exclusion to
  avoid overcounting.  Using a decomposition of $h$ into powers of linear
  forms, then the method of the previous theorem gives the
  result. 
\end{proof}

\section*{Acknowledgments}

 This article is part of a research which was made possible by
 several meetings of the authors,
at the Centro di Ricerca Matematica Ennio De Giorgi of the Scuola
Normale Superiore, Pisa in 2009,  in a SQuaRE program at the
American Institute of Mathematics, Palo Alto, in July 2009 and
September 2010, and in the Research in Pairs program at
Mathematisches Forschungsinstitut Oberwolfach in March/April 2010.
The support of all three institutions is gratefully acknowledged.

V.~Baldoni was partially supported by the Cofin 40\%, MIUR.

M.~K\"oppe was partially supported by grant DMS-0914873 of the National
Science Foundation.

\clearpage
\bibliographystyle{../amsabbrv}
\bibliography{../biblio}

\end{document}

%% file: projected-lattice.pdf_t
\begin{picture}(0,0)%
\includegraphics{projected-lattice.pdf}%
\end{picture}%
\setlength{\unitlength}{4144sp}%
\begingroup\makeatletter\ifx\SetFigFontNFSS\undefined%
\gdef\SetFigFontNFSS#1#2#3#4#5{%
  \reset@font\fontsize{#1}{#2pt}%
  \fontfamily{#3}\fontseries{#4}\fontshape{#5}%
  \selectfont}%
\fi\endgroup%
\begin{picture}(2904,2184)(-326,197)
\put(676,1649){\makebox(0,0)[b]{\smash{{\SetFigFontNFSS{12}{14.4}{\familydefault}{\mddefault}{\updefault}{\color[rgb]{0,0,.56}$v_2$}%
}}}}
\put(856,601){\makebox(0,0)[b]{\smash{{\SetFigFontNFSS{12}{14.4}{\familydefault}{\mddefault}{\updefault}{\color[rgb]{0,0,.56}$v_1$}%
}}}}
\end{picture}%

%% file: projected-lattice.pstex_t
\begin{picture}(0,0)%
\includegraphics{projected-lattice.pstex}%
\end{picture}%
\setlength{\unitlength}{4144sp}%
\begingroup\makeatletter\ifx\SetFigFontNFSS\undefined%
\gdef\SetFigFontNFSS#1#2#3#4#5{%
  \reset@font\fontsize{#1}{#2pt}%
  \fontfamily{#3}\fontseries{#4}\fontshape{#5}%
  \selectfont}%
\fi\endgroup%
\begin{picture}(2904,2184)(-326,197)
\put(676,1649){\makebox(0,0)[b]{\smash{{\SetFigFontNFSS{12}{14.4}{\familydefault}{\mddefault}{\updefault}{\color[rgb]{0,0,.56}$v_2$}%
}}}}
\put(856,601){\makebox(0,0)[b]{\smash{{\SetFigFontNFSS{12}{14.4}{\familydefault}{\mddefault}{\updefault}{\color[rgb]{0,0,.56}$v_1$}%
}}}}
\end{picture}%

%% file: brionvergne-2d.pdf_t
\begin{picture}(0,0)%
\includegraphics{brionvergne-2d.pdf}%
\end{picture}%
\setlength{\unitlength}{4144sp}%
\begingroup\makeatletter\ifx\SetFigFontNFSS\undefined%
\gdef\SetFigFontNFSS#1#2#3#4#5{%
  \reset@font\fontsize{#1}{#2pt}%
  \fontfamily{#3}\fontseries{#4}\fontshape{#5}%
  \selectfont}%
\fi\endgroup%
\begin{picture}(2702,2252)(-1350,-287)
\put(631,-106){\makebox(0,0)[lb]{\smash{{\SetFigFontNFSS{10}{12.0}{\familydefault}{\mddefault}{\updefault}{\color[rgb]{0,.69,0}$a$}%
}}}}
\put(856,1041){\makebox(0,0)[lb]{\smash{{\SetFigFontNFSS{10}{12.0}{\familydefault}{\mddefault}{\updefault}{\color[rgb]{0,0,0}$\c$}%
}}}}
\put(901,1559){\makebox(0,0)[lb]{\smash{{\SetFigFontNFSS{10}{12.0}{\familydefault}{\mddefault}{\updefault}{\color[rgb]{1,0,0}$L$}%
}}}}
\put(-944,929){\makebox(0,0)[lb]{\smash{{\SetFigFontNFSS{10}{12.0}{\familydefault}{\mddefault}{\updefault}{\color[rgb]{0,0,0}$w_2\bmod L$}%
}}}}
\put(-269,299){\makebox(0,0)[lb]{\smash{{\SetFigFontNFSS{10}{12.0}{\familydefault}{\mddefault}{\updefault}{\color[rgb]{0,0,0}$w_1\bmod L$}%
}}}}
\put(-944,1289){\makebox(0,0)[lb]{\smash{{\SetFigFontNFSS{10}{12.0}{\familydefault}{\mddefault}{\updefault}{\color[rgb]{0,0,0}$V/L$}%
}}}}
\put(451,1514){\makebox(0,0)[lb]{\smash{{\SetFigFontNFSS{10}{12.0}{\familydefault}{\mddefault}{\updefault}{\color[rgb]{1,0,0}$v$}%
}}}}
\put(-224,1514){\makebox(0,0)[lb]{\smash{{\SetFigFontNFSS{10}{12.0}{\familydefault}{\mddefault}{\updefault}{\color[rgb]{0,0,1}$w_2$}%
}}}}
\put(697,678){\makebox(0,0)[lb]{\smash{{\SetFigFontNFSS{10}{12.0}{\familydefault}{\mddefault}{\updefault}{\color[rgb]{0,0,1}$w_1$}%
}}}}
\end{picture}%

%% file: brionvergne-2d.pstex_t
\begin{picture}(0,0)%
\includegraphics{brionvergne-2d.pstex}%
\end{picture}%
\setlength{\unitlength}{4144sp}%
\begingroup\makeatletter\ifx\SetFigFontNFSS\undefined%
\gdef\SetFigFontNFSS#1#2#3#4#5{%
  \reset@font\fontsize{#1}{#2pt}%
  \fontfamily{#3}\fontseries{#4}\fontshape{#5}%
  \selectfont}%
\fi\endgroup%
\begin{picture}(2702,2252)(-1350,-287)
\put(631,-106){\makebox(0,0)[lb]{\smash{{\SetFigFontNFSS{10}{12.0}{\familydefault}{\mddefault}{\updefault}{\color[rgb]{0,.69,0}$a$}%
}}}}
\put(856,1041){\makebox(0,0)[lb]{\smash{{\SetFigFontNFSS{10}{12.0}{\familydefault}{\mddefault}{\updefault}{\color[rgb]{0,0,0}$\c$}%
}}}}
\put(901,1559){\makebox(0,0)[lb]{\smash{{\SetFigFontNFSS{10}{12.0}{\familydefault}{\mddefault}{\updefault}{\color[rgb]{1,0,0}$L$}%
}}}}
\put(-944,929){\makebox(0,0)[lb]{\smash{{\SetFigFontNFSS{10}{12.0}{\familydefault}{\mddefault}{\updefault}{\color[rgb]{0,0,0}$w_2\bmod L$}%
}}}}
\put(-269,299){\makebox(0,0)[lb]{\smash{{\SetFigFontNFSS{10}{12.0}{\familydefault}{\mddefault}{\updefault}{\color[rgb]{0,0,0}$w_1\bmod L$}%
}}}}
\put(-944,1289){\makebox(0,0)[lb]{\smash{{\SetFigFontNFSS{10}{12.0}{\familydefault}{\mddefault}{\updefault}{\color[rgb]{0,0,0}$V/L$}%
}}}}
\put(451,1514){\makebox(0,0)[lb]{\smash{{\SetFigFontNFSS{10}{12.0}{\familydefault}{\mddefault}{\updefault}{\color[rgb]{1,0,0}$v$}%
}}}}
\put(-224,1514){\makebox(0,0)[lb]{\smash{{\SetFigFontNFSS{10}{12.0}{\familydefault}{\mddefault}{\updefault}{\color[rgb]{0,0,1}$w_2$}%
}}}}
\put(697,678){\makebox(0,0)[lb]{\smash{{\SetFigFontNFSS{10}{12.0}{\familydefault}{\mddefault}{\updefault}{\color[rgb]{0,0,1}$w_1$}%
}}}}
\end{picture}%

%% file: brionvergne-2d-cone1.pdf_t
\begin{picture}(0,0)%
\includegraphics{brionvergne-2d-cone1.pdf}%
\end{picture}%
\setlength{\unitlength}{4144sp}%
\begingroup\makeatletter\ifx\SetFigFontNFSS\undefined%
\gdef\SetFigFontNFSS#1#2#3#4#5{%
  \reset@font\fontsize{#1}{#2pt}%
  \fontfamily{#3}\fontseries{#4}\fontshape{#5}%
  \selectfont}%
\fi\endgroup%
\begin{picture}(2477,2341)(-1125,-287)
\put(631,-106){\makebox(0,0)[lb]{\smash{{\SetFigFontNFSS{10}{12.0}{\familydefault}{\mddefault}{\updefault}{\color[rgb]{0,.69,0}$a$}%
}}}}
\put(901,1559){\makebox(0,0)[lb]{\smash{{\SetFigFontNFSS{10}{12.0}{\familydefault}{\mddefault}{\updefault}{\color[rgb]{1,0,0}$L$}%
}}}}
\put(-269,299){\makebox(0,0)[lb]{\smash{{\SetFigFontNFSS{10}{12.0}{\familydefault}{\mddefault}{\updefault}{\color[rgb]{0,.69,0}$w_1\bmod L$}%
}}}}
\put(-359,1919){\makebox(0,0)[lb]{\smash{{\SetFigFontNFSS{10}{12.0}{\familydefault}{\mddefault}{\updefault}{\color[rgb]{0,0,0}$\sigma=\{1\}$}%
}}}}
\put(774,1299){\makebox(0,0)[lb]{\smash{{\SetFigFontNFSS{10}{12.0}{\familydefault}{\mddefault}{\updefault}{\color[rgb]{0,0,0}$\c_\sigma$}%
}}}}
\put(-224,1514){\makebox(0,0)[lb]{\smash{{\SetFigFontNFSS{10}{12.0}{\familydefault}{\mddefault}{\updefault}{\color[rgb]{0,0,1}$w_2$}%
}}}}
\put(384,1069){\makebox(0,0)[lb]{\smash{{\SetFigFontNFSS{10}{12.0}{\familydefault}{\mddefault}{\updefault}{\color[rgb]{1,0,0}$\rho_{\{1\}}(w_2)$}%
}}}}
\put(650,678){\makebox(0,0)[lb]{\smash{{\SetFigFontNFSS{10}{12.0}{\familydefault}{\mddefault}{\updefault}{\color[rgb]{0,.69,0}$\epsilon_{\{1\},1}$}%
}}}}
\put(1028,678){\makebox(0,0)[lb]{\smash{{\SetFigFontNFSS{10}{12.0}{\familydefault}{\mddefault}{\updefault}{\color[rgb]{0,0,1}$w_1$}%
}}}}
\end{picture}%

%% file: brionvergne-2d-cone1.pstex_t
\begin{picture}(0,0)%
\includegraphics{brionvergne-2d-cone1.pstex}%
\end{picture}%
\setlength{\unitlength}{4144sp}%
\begingroup\makeatletter\ifx\SetFigFontNFSS\undefined%
\gdef\SetFigFontNFSS#1#2#3#4#5{%
  \reset@font\fontsize{#1}{#2pt}%
  \fontfamily{#3}\fontseries{#4}\fontshape{#5}%
  \selectfont}%
\fi\endgroup%
\begin{picture}(2477,2341)(-1125,-287)
\put(631,-106){\makebox(0,0)[lb]{\smash{{\SetFigFontNFSS{10}{12.0}{\familydefault}{\mddefault}{\updefault}{\color[rgb]{0,.69,0}$a$}%
}}}}
\put(901,1559){\makebox(0,0)[lb]{\smash{{\SetFigFontNFSS{10}{12.0}{\familydefault}{\mddefault}{\updefault}{\color[rgb]{1,0,0}$L$}%
}}}}
\put(-269,299){\makebox(0,0)[lb]{\smash{{\SetFigFontNFSS{10}{12.0}{\familydefault}{\mddefault}{\updefault}{\color[rgb]{0,.69,0}$w_1\bmod L$}%
}}}}
\put(-359,1919){\makebox(0,0)[lb]{\smash{{\SetFigFontNFSS{10}{12.0}{\familydefault}{\mddefault}{\updefault}{\color[rgb]{0,0,0}$\sigma=\{1\}$}%
}}}}
\put(774,1299){\makebox(0,0)[lb]{\smash{{\SetFigFontNFSS{10}{12.0}{\familydefault}{\mddefault}{\updefault}{\color[rgb]{0,0,0}$\c_\sigma$}%
}}}}
\put(-224,1514){\makebox(0,0)[lb]{\smash{{\SetFigFontNFSS{10}{12.0}{\familydefault}{\mddefault}{\updefault}{\color[rgb]{0,0,1}$w_2$}%
}}}}
\put(384,1069){\makebox(0,0)[lb]{\smash{{\SetFigFontNFSS{10}{12.0}{\familydefault}{\mddefault}{\updefault}{\color[rgb]{1,0,0}$\rho_{\{1\}}(w_2)$}%
}}}}
\put(650,678){\makebox(0,0)[lb]{\smash{{\SetFigFontNFSS{10}{12.0}{\familydefault}{\mddefault}{\updefault}{\color[rgb]{0,.69,0}$\epsilon_{\{1\},1}$}%
}}}}
\put(1028,678){\makebox(0,0)[lb]{\smash{{\SetFigFontNFSS{10}{12.0}{\familydefault}{\mddefault}{\updefault}{\color[rgb]{0,0,1}$w_1$}%
}}}}
\end{picture}%

%% file: brionvergne-2d-cone2.pdf_t
\begin{picture}(0,0)%
\includegraphics{brionvergne-2d-cone2.pdf}%
\end{picture}%
\setlength{\unitlength}{4144sp}%
\begingroup\makeatletter\ifx\SetFigFontNFSS\undefined%
\gdef\SetFigFontNFSS#1#2#3#4#5{%
  \reset@font\fontsize{#1}{#2pt}%
  \fontfamily{#3}\fontseries{#4}\fontshape{#5}%
  \selectfont}%
\fi\endgroup%
\begin{picture}(2477,2341)(-1125,-287)
\put(631,-106){\makebox(0,0)[lb]{\smash{{\SetFigFontNFSS{10}{12.0}{\familydefault}{\mddefault}{\updefault}{\color[rgb]{0,.69,0}$a$}%
}}}}
\put(901,1559){\makebox(0,0)[lb]{\smash{{\SetFigFontNFSS{10}{12.0}{\familydefault}{\mddefault}{\updefault}{\color[rgb]{1,0,0}$L$}%
}}}}
\put(774,408){\makebox(0,0)[lb]{\smash{{\SetFigFontNFSS{10}{12.0}{\familydefault}{\mddefault}{\updefault}{\color[rgb]{0,0,0}$\c_\sigma$}%
}}}}
\put(-359,1919){\makebox(0,0)[lb]{\smash{{\SetFigFontNFSS{10}{12.0}{\familydefault}{\mddefault}{\updefault}{\color[rgb]{0,0,0}$\sigma=\{2\}$}%
}}}}
\put(-944,929){\makebox(0,0)[lb]{\smash{{\SetFigFontNFSS{10}{12.0}{\familydefault}{\mddefault}{\updefault}{\color[rgb]{0,.69,0}$w_2\bmod L$}%
}}}}
\put(-610, 48){\makebox(0,0)[lb]{\smash{{\SetFigFontNFSS{10}{12.0}{\familydefault}{\mddefault}{\updefault}{\color[rgb]{0,.69,0}$\epsilon_{\{2\},2}$}%
}}}}
\put(-232, 48){\makebox(0,0)[lb]{\smash{{\SetFigFontNFSS{10}{12.0}{\familydefault}{\mddefault}{\updefault}{\color[rgb]{0,0,1}$w_2$}%
}}}}
\put(-224,1514){\makebox(0,0)[lb]{\smash{{\SetFigFontNFSS{10}{12.0}{\familydefault}{\mddefault}{\updefault}{\color[rgb]{0,0,1}$w_2$}%
}}}}
\put(488,1127){\makebox(0,0)[lb]{\smash{{\SetFigFontNFSS{10}{12.0}{\familydefault}{\mddefault}{\updefault}{\color[rgb]{1,0,0}$\rho_{\{2\}}(w_1)$}%
}}}}
\end{picture}%

%% file: brionvergne-2d-cone2.pstex_t
\begin{picture}(0,0)%
\includegraphics{brionvergne-2d-cone2.pstex}%
\end{picture}%
\setlength{\unitlength}{4144sp}%
\begingroup\makeatletter\ifx\SetFigFontNFSS\undefined%
\gdef\SetFigFontNFSS#1#2#3#4#5{%
  \reset@font\fontsize{#1}{#2pt}%
  \fontfamily{#3}\fontseries{#4}\fontshape{#5}%
  \selectfont}%
\fi\endgroup%
\begin{picture}(2477,2341)(-1125,-287)
\put(631,-106){\makebox(0,0)[lb]{\smash{{\SetFigFontNFSS{10}{12.0}{\familydefault}{\mddefault}{\updefault}{\color[rgb]{0,.69,0}$a$}%
}}}}
\put(901,1559){\makebox(0,0)[lb]{\smash{{\SetFigFontNFSS{10}{12.0}{\familydefault}{\mddefault}{\updefault}{\color[rgb]{1,0,0}$L$}%
}}}}
\put(774,408){\makebox(0,0)[lb]{\smash{{\SetFigFontNFSS{10}{12.0}{\familydefault}{\mddefault}{\updefault}{\color[rgb]{0,0,0}$\c_\sigma$}%
}}}}
\put(-359,1919){\makebox(0,0)[lb]{\smash{{\SetFigFontNFSS{10}{12.0}{\familydefault}{\mddefault}{\updefault}{\color[rgb]{0,0,0}$\sigma=\{2\}$}%
}}}}
\put(-944,929){\makebox(0,0)[lb]{\smash{{\SetFigFontNFSS{10}{12.0}{\familydefault}{\mddefault}{\updefault}{\color[rgb]{0,.69,0}$w_2\bmod L$}%
}}}}
\put(-610, 48){\makebox(0,0)[lb]{\smash{{\SetFigFontNFSS{10}{12.0}{\familydefault}{\mddefault}{\updefault}{\color[rgb]{0,.69,0}$\epsilon_{\{2\},2}$}%
}}}}
\put(-232, 48){\makebox(0,0)[lb]{\smash{{\SetFigFontNFSS{10}{12.0}{\familydefault}{\mddefault}{\updefault}{\color[rgb]{0,0,1}$w_2$}%
}}}}
\put(-224,1514){\makebox(0,0)[lb]{\smash{{\SetFigFontNFSS{10}{12.0}{\familydefault}{\mddefault}{\updefault}{\color[rgb]{0,0,1}$w_2$}%
}}}}
\put(488,1127){\makebox(0,0)[lb]{\smash{{\SetFigFontNFSS{10}{12.0}{\familydefault}{\mddefault}{\updefault}{\color[rgb]{1,0,0}$\rho_{\{2\}}(w_1)$}%
}}}}
\end{picture}%